\documentclass[10pt]{article}
\usepackage[latin1]{inputenc}
\usepackage[T1]{fontenc}

\usepackage[bitstream-charter]{mathdesign}

\usepackage{hyperref}
\hypersetup{
    hyperfootnotes=false,
    colorlinks=true,
    linkcolor=black,
    citecolor=black,
    filecolor=black,
    urlcolor=black
}

\usepackage[usenames,dvipsnames,svgnames,table]{xcolor}
\usepackage{amsmath}
\usepackage{amsthm}
\swapnumbers
\usepackage[alphabetic]{amsrefs}
\usepackage[all]{xy}
\usepackage{mathtools}
\usepackage{stmaryrd}
\usepackage{enumitem}
\usepackage{fixltx2e}
\usepackage[defaultlines=3,all]{nowidow}
\usepackage{tabularx}

\usepackage{geometry}

\usepackage{microtype}

\theoremstyle{plain}
\newtheorem{thm}{Theorem}[subsection]
\newtheorem{lma}[thm]{Lemma}
\newtheorem{cor}[thm]{Corollary}
\newtheorem{prop}[thm]{Proposition}
\newtheorem*{announcethm}{Theorem}

\theoremstyle{definition}
\newtheorem{dfn}[thm]{Definition}

\theoremstyle{remark}
\newtheorem{rem}[thm]{Remark}
\newtheorem{ex}[thm]{Example}
\newtheorem{conv}[thm]{Convention}
\newtheorem{notn}[thm]{Notation}

\DeclareMathOperator{\Aut}{Aut}
\DeclareMathOperator{\Chow}{CH}
\DeclareMathOperator{\chow}{ch}
\DeclareMathOperator{\codim}{codim}
\DeclareMathOperator{\coker}{coker}

\DeclareMathOperator{\cone}{cone}
\DeclareMathOperator{\Cyc}{Z}
\DeclareMathOperator{\Db}{D^b}
\DeclareMathOperator{\divis}{div}
\DeclareMathOperator{\Dperf}{D^{perf}}

\DeclareMathOperator{\Ext}{Ext}
\DeclareMathOperator{\Hom}{Hom}
\DeclareMathOperator{\id}{id}
\DeclareMathOperator{\im}{im}
\DeclareMathOperator{\Ind}{Ind}
\DeclareMathOperator{\Kzero}{K_0}
\DeclareMathOperator{\length}{length}
\DeclareMathOperator{\Min}{Min}
\DeclareMathOperator{\Mod}{\!-mod}

\DeclareMathOperator{\proj}{\!-proj}
\DeclareMathOperator{\Proj}{Proj}

\DeclareMathOperator{\Res}{Res}
\DeclareMathOperator{\Spc}{Spc}
\DeclareMathOperator{\stab}{\!-stab}
\DeclareMathOperator{\supp}{supp}

\begin{document}
\title{Chow groups of tensor triangulated categories}
\author{Sebastian Klein}
\date{}

\maketitle

\begin{abstract}
	We recall P. Balmer's definition of tensor triangular Chow group for a tensor triangulated category $\mathcal{K}$ and explore some of its properties. We give a proof that for a suitably nice scheme $X$ it recovers the usual notion of Chow group from algebraic geometry when we put~$\mathcal{K} = \Dperf(X)$. Furthermore, we identify a class of functors for which tensor triangular Chow groups behave functorially and show that (for suitably nice schemes) proper push-forward and flat pull-back of algebraic cycles can be interpreted as being induced by the derived inverse and direct image functors between the bounded derived categories of the involved schemes. We also compute some examples for derived and stable categories from modular representation theory, where we obtain tensor triangular cycle groups with torsion coefficients. This illustrates our point of view that tensor triangular cycles are elements of a certain Grothendieck group, rather than $\mathbb{Z}$-linear combinations of closed subspaces of some topological space.
\end{abstract}

\tableofcontents

\pagestyle{headings}
\section{Introduction}
A basic topic in algebraic geometry is the study of algebraic cycles on a variety $X$ under the equivalence relation of rational equivalence. This is usually formalized by the Chow group 
\[\Chow(X) = \bigoplus_p \Chow^p(X)\]
where $\Chow^p(X)$ is the free abelian group on subvarieties $Y \subset X$ of codimension $p$, modulo the subgroup of cycles rationally equivalent to zero (i.e.\ those that appear as the divisor of a rational function on a subvariety of codimension $p-1$).

How should one approach the subject from the point of view of the derived category of~$X$? In \cite{balmer2005spectrum}, it is shown that we can reconstruct $X$ from its derived category of perfect complexes $\Dperf(X)$ considered as a \emph{tensor triangulated category}. Thus, it should also be possible to reconstruct $\Chow(X)$ from $\Dperf(X)$ ``in purely \emph{tensor triangular} terms''. More precisely one would like to construct for each $p \geq 0$ a functor $\Chow^{\Delta}_p(-)$, that takes a tensor triangulated category $\mathcal{K}$ and produces a group $\Chow^{\Delta}_p(\mathcal{K})$ such that~$\Chow^{\Delta}_p(\Dperf(X)) \cong \Chow^p(X)$.

In this article we show that such a construction is given by a definition of $\Chow^{\Delta}_p(-)$ suggested to the author by P. Balmer in 2011 and now published in \cite{balmerchow}. The starting point here is the observation that when one filters the category $\Dperf(X)$ by codimension of support, the successive subquotients split as a coproduct of ``local categories'' (cf. \cite{balmerfiltrations}), analogously to what happens when one performs the same procedure for the abelian category~$\mathrm{Coh}(X)$. One continues to define the codimension $p$ cycle group of $\Dperf(X)$ as the Grothendieck group of the $p$-th subquotient of the filtration. We then obtain a definition of the codimension-$p$ Chow group of $\Dperf(X)$ by analogy with Quillen's coniveau spectral sequence (see \cite{quillenhigher}*{\S 7}). We prove:
\begin{announcethm}[\ref{agreementthm}]
	Let $X$ be a non-singular, separated scheme of finite type over a field. Endow $\Dperf(X)$ with the opposite of the Krull codimension as a dimension function (cf. Definition \ref{dimfuncdef}). Then for all $p \in \mathbb{Z}$,
	\[ \Chow^{\Delta}_p\left(\Dperf(X)\right) \cong \Chow^{-p}(X) ~.\]
\end{announcethm}

Apart from reconstructing the classical Chow groups, the definition of $\Chow^{\Delta}_p(\mathcal{K})$ also behaves well in its own right, when we consider it as an invariant of $\mathcal{K}$.  We show that $\Chow^{\Delta}_p(-)$ is functorial for the class of exact functors with a relative dimension (see Definition \ref{funcreldimdef}). These are exact (\emph{not necessarily} $\otimes$-exact) functors that preserve the filtration that the choice of a dimension function induces on $\mathcal{T}$, up to a shift by $n$. We have
\begin{announcethm}[\ref{functoriality}]
	Let $F: \mathcal{K} \to \mathcal{L}$ be an exact functor of relative dimension $n$. Then for all $p \in \mathbb{Z}$, $F$ induces a group homomorphism
	\[\Chow^{\Delta}_p(F): \Chow^{\Delta}_p(\mathcal{K}) \to \Chow^{\Delta}_{p+n}(\mathcal{L})\]
\end{announcethm}
\noindent and we prove that the proper push-forward and flat pull-back morphisms on the classical Chow groups of non-singular varieties can be interpreted as special cases of the above theorem (Theorem \ref{pullbackagrees} and Theorem \ref{proppushforwardagrees}).

After investigating a different notion of rational equivalence in Section \ref{sectaltrat}, we then apply our theory outside of algebraic geometry and proceed to compute some examples from modular representation theory: for a finite group $G$ and a field $k$ whose characteristic divides $|G|$, we look at the bounded derived category $\Db(kG\Mod)$ and the stable module category $kG\stab$. Both categories are naturally tensor triangulated with tensor product $\otimes_k$ and we show that they have isomorphic tensor triangular Chow groups in almost all degrees, which should not come as a big surprise in view of Rickard's equivalence (see \cite{rickardstable}) 
\[kG\stab \cong \Db(kG\Mod)/\Dperf(kG\Mod)~.\]
We prove:
\begin{announcethm}[see Theorem \ref{thmderivedstabiso}]
	Consider $kG\stab$ and $\Db(kG\Mod)$ with the Krull dimension of support as a dimension function on $\Spc(kG\stab)$ and $\Spc(\Db(kG\Mod))$. Then for all $p \geq 0$, there are isomorphisms
	\[\Chow^{\Delta}_p(kG\stab) \cong \Chow^{\Delta}_{p+1}(\Db(kG\Mod))~.\]
\end{announcethm}

We then continue to compute the associated tensor triangular Chow groups for~$G = \mathbb{Z}/ p^n \mathbb{Z}$ and $G= \mathbb{Z}/2\mathbb{Z} \times \mathbb{Z}/2\mathbb{Z}$:
\begin{announcethm}[see Propositions \ref{cyclicexample}, \ref{Kleinfourexample1} and \ref{Kleinfourexample2}] Let $k$ be a field of characteristic $p$. For $G = \mathbb{Z}/ p^n \mathbb{Z}$, we have
	\begin{enumerate}[label=(\roman*)]
		\item $\Chow^{\Delta}_i (kG\stab) = 0 \quad \forall i \neq 0$,
		\item $\Chow^{\Delta}_0 (kG\stab) \cong \mathbb{Z}/ p^n \mathbb{Z}$,
	\end{enumerate}
	and if $p=2$ and $H=\mathbb{Z}/2\mathbb{Z} \times \mathbb{Z}/2\mathbb{Z}$ then
	\begin{enumerate}[resume,label=(\roman*)]
		\item $\Chow^{\Delta}_i (kH\stab) = 0 \quad \forall i \neq 0,1$,
		\item $\Chow^{\Delta}_0 (kH\stab) \cong \mathbb{Z}/ 2 \mathbb{Z}$,
		\item $\Chow^{\Delta}_1 (kH\stab) \cong  \mathbb{Z}/ 2 \mathbb{Z}$ if $k$ is algebraically closed,
	\end{enumerate}
	when we endow both $kG\stab$ and $kH\stab$ with the Krull dimension as a dimension function.
\end{announcethm}
In the course of the above computations, we also see that it is possible to obtain cycle groups with torsion coefficients (see Proposition \ref{cyclicexample}), which contrasts with the situation in the algebro-geometric case. This illustrates that we view a general cycle, rather than as a $\mathbb{Z}$-linear combination of irreducible subspaces of codimension $p$ of the spectrum $\Spc(\mathcal{K})$, as an element of a Grothendieck group $\Kzero(\mathcal{K}_{(p)}/\mathcal{K}_{(p-1)})$. Only in the non-singular algebro-geometric examples does this produce coefficients in $\mathbb{Z}$, due to the ``coincidence'' that the Grothendieck group of the bounded derived category of finite-length modules over a local ring is isomorphic to~$\mathbb{Z}$. We finish the representation-theoretic part of our story by showing that  stable restriction and induction functors fit into the framework of functors with a relative dimension (see Theorems \ref{thmresreldim0}, \ref{thmreldimind0}).\\[0.5cm]
\textbf{Acknowledgments:} With the exception of some minor changes, the results in this paper are all contained in the author's Ph.D.\ thesis, which was written at Utrecht University and jointly supervised by Paul Balmer and Gunther Cornelissen. The author would like to thank both of them for their support. The members of the thesis committee and an anonymous referee also pointed out some improvements and corrections, for which the author is grateful. The author acknowledges the support of the European Union for the ERC grant No.\ 257004-HHNcdMir.

\section{Definitions and conventions}
Let us recall some of the theory that we will use in the following.

\subsection{Chow groups in algebraic geometry}
We aim at generalizing the study of cycles on an algebraic variety modulo rational equivalence. The basic setup of this theory is as follows: for an algebraic variety $X$ (by which we shall mean a separated scheme of finite type over a field) one looks for each $p \geq 0$ at the \emph{codimension $p$ cycle group} $\Cyc^p(X)$, the free abelian group on subvarieties (= closed integral subschemes) of codimension $p$ in $X$. One now introduces an equivalence relation on $\Cyc^p(X)$. Two cycles in $\Cyc^p(X)$ are considered \emph{rationally equivalent} if there exists a finite number of subvarieties $Y_i \subset X$ of codimension $p-1$ and elements of the function fields $f_i \in \mathrm{K}(Y_i)$ such that the difference of the two cycles is equal to the sum of the cycles $\divis(f_i)$, the divisors associated to the functions $f_i$. The divisors $\divis(f_i)$ should be thought of as the sum of the zeroes of $f_i$ minus the sum of the poles of $f_i$, both counted with multiplicities (see \cite{fulton} for the formal definition). The cycles rationally equivalent to zero form a subgroup of $\Cyc^p(X)$ and the corresponding quotient is $\Chow^p(X)$, the \emph{codimension $p$ Chow group of $X$}.

\subsection{Tensor triangular geometry}
The main objects of study in this article are tensor triangulated categories. We refer the reader to \cite{neemantc} for the basic theory of triangulated categories.
\begin{dfn}[see \cite{balmer2010tensor}*{Definition 3}]
	A tensor triangulated category is an essentially small triangulated category $\mathcal{K}$ endowed with a compatible symmetric monoidal structure. That is, there is a bifunctor
	\[\otimes: \mathcal{K} \times \mathcal{K} \to \mathcal{K}\]
	and a unit object $\mathbb{I}$, together with associator, unitor and commutator isomorphisms: 
	for all objects $X,Y,Z$ in $\mathcal{T}$, we have natural isomorphisms 
	\[X \otimes (Y \otimes Z) \cong (X \otimes Y) \otimes Z, \quad X \otimes \mathbb{I} \cong X \cong \mathbb{I} \otimes X, \quad X \otimes Y \cong Y \otimes X\]
	that satisfy the coherence conditions of \cite{maclanecwm}*{Section XI.1} to make $\mathcal{T}$ a symmetric monoidal category. Furthermore, the bifunctor $\otimes$ is exact in each variable.
\label{defttcat}
\end{dfn}

\begin{rem}
	As an addition to the axiomatic of Definition \ref{defttcat}, one can ask that the following coherence condition holds: by the biexactness of $\otimes$, we have natural isomorphisms $(\Sigma X) \otimes - \cong \Sigma(X \otimes -)$ and $- \otimes \Sigma(Y) \cong \Sigma(- \otimes Y)$ for all objects $X,Y \in \mathcal{T}$. These fit into a diagram
	\[
	\xymatrix{
		(\Sigma X) \otimes (\Sigma Y) \ar[r]^{\sim} \ar[d]^{\wr} & \Sigma(X \otimes \Sigma Y) \ar[d]^{\wr} \\
		\Sigma(\Sigma X \otimes Y) \ar[r]^{\sim} & \Sigma^2(X \otimes Y)
	}
	\]
	which one requires to commute up to a sign, i.e. the composition of the upper and right isomorphisms should equal the composition of the left and lower isomorphisms or its additive inverse (see e.g.\ \cite{balmer2010tensor}). 
	\label{remcoherenceprodcond}
\end{rem}

\begin{dfn}
	A triangulated category $\mathcal{K}$ is called \emph{idempotent complete} if all idempotent endomorphisms in $\mathcal{K}$ split. A full subcategory $\mathcal{J} \subset \mathcal{K}$ is called \emph{dense} if every object of $\mathcal{K}$ is a direct summand of an object of $\mathcal{J}$.
\end{dfn}
Recall that we can always embed a (tensor) triangulated category $\mathcal{K}$ into its \emph{idempotent completion} $\mathcal{K}^{\natural}$ as a dense subcategory. The category $\mathcal{K}^{\natural}$ is idempotent complete and naturally (tensor) triangulated, and the essential image of $\mathcal{K}$ in $\mathcal{K}^{\natural}$ is dense (see \cite{balschlichidem} and \cite{balmer2005spectrum}*{Remark 3.12}).

The main object one studies in \emph{tensor triangular geometry} is the \emph{spectrum} of a tensor triangulated category, whose construction we briefly describe next.
\begin{dfn} \index{Tensor ideal@$\otimes$-ideal} 
	Let $\mathcal{K}$ be a tensor triangulated category. A thick triangulated subcategory $\mathcal{J} \subset \mathcal{K}$ is called
	\begin{itemize}
		\item \emph{$\otimes$-ideal} if $\mathcal{K} \otimes \mathcal{J} \subset \mathcal{J}$.
		\item \emph{prime} if $\mathcal{J}$ is a proper $\otimes$-ideal ($\mathcal{J} \neq \mathcal{K}$) and $A \otimes B \in \mathcal{J}$ implies $A \in \mathcal{J}$ or $B \in \mathcal{J}$ for all objects $A,B \in \mathcal{K}$.
	\end{itemize}
\end{dfn}

\begin{dfn}[see \cite{balmer2005spectrum}] \index{Spectrum of a tensor triangulated category} 
	Let $\mathcal{K}$ be an essentially small tensor triangulated category. The \emph{spectrum} of $\mathcal{K}$ is the set
	\[\Spc(\mathcal{K}) := \lbrace \mathcal{P} \subset \mathcal{K}: \mathcal{P}~\text{is a prime ideal} \rbrace\]
	topologized by the basis of closed sets of the form
	\[\supp(A) := \lbrace \mathcal{P} \in \Spc(\mathcal{K}): A \notin \mathcal{P} \rbrace\]
	for objects $A \in \mathcal{K}$. The set $\supp(A)$ is called the \emph{support of $A$}.
\end{dfn}

Let us give some computations of $\Spc(\mathcal{K})$:
\begin{ex}[see \cite{balmer2005spectrum}, \cite{bkssupport}*{Theorem 9.5}]
	Let $X$ be a quasi-compact, quasi-sep\-a\-rated scheme, then $X \cong \Spc(\Dperf(X))$.  The homeomorphism is given as follows:
	\begin{align*}
		\sigma: X &\to \Spc(\Dperf(X)) \\
		x & \mapsto \left\lbrace A^{\bullet} \in \Dperf(X)| A^{\bullet}_x \cong 0 \text{ in } \Dperf(\mathcal{O}_{X,x})\right\rbrace \subset \Dperf(X)
	\end{align*}
	Moreover, the support $\supp(A^{\bullet})$ of a complex $A^{\bullet} \in \Dperf(X)$ coincides with the support of the total homology of $A^{\bullet}$ on $X$ under this isomorphism. The proof of these statements uses Thomason's classification result from \cite{thomasonclassification}.
	\label{exschemereconstruct}
\end{ex}

\begin{ex}[see \cite{balmer2005spectrum}*{Corollary 5.10}]
	Let $G$ be a finite group and $k$ be a field such that $\mathrm{char}(k)$ divides the order of $G$. Then $\Spc(kG\stab) \cong \mathcal{V}_G(k)$, the \emph{projective support variety} of $k$. The variety $\mathcal{V}_G(k)$ is defined as $\Proj(\mathrm{H}^*(G,k))$, where $\mathrm{H}^*(G,k)$ denotes the cohomology ring of $G$ over $k$. The support $\supp(M)$ of a module $M \in kG\stab$ coincides with the cohomological support of $M$ in $\mathcal{V}_G(k)$ under this isomorphism. The proof of these statements uses the classification of thick $\otimes$-ideals in $kG\stab$ from \cite{bencarlrick}.
\end{ex}

The spectrum of a tensor triangulated category $\mathcal{K}$ gives us the possibility to assign to each object of $\mathcal{K}$ the dimension of its support.
\begin{dfn}[see \cite{balmerfiltrations}] 
	A \emph{dimension function} on $\mathcal{K}$ is a map 
	\[\dim: \Spc(\mathcal{K}) \to \mathbb{Z} \cup \lbrace \pm \infty \rbrace\]
	such that the following two conditions hold:
	\begin{enumerate}
		\item If $\mathcal{Q} \subset \mathcal{P}$ are prime tensor ideals of $\mathcal{K}$, then~$\dim(\mathcal{Q}) \leq \dim(\mathcal{P})$.
		\item If $\mathcal{Q} \subset \mathcal{P}$ and $\dim(\mathcal{Q}) = \dim(\mathcal{P}) \in \mathbb{Z}$, then~$\mathcal{Q} = \mathcal{P}$.
	\end{enumerate}
	For a subset $V \subset \Spc(\mathcal{K})$, we define $\dim(V) := \sup \lbrace \dim(\mathcal{P}) | \mathcal{P} \in V \rbrace$. For every $p \in \mathbb{Z} \cup \lbrace \pm \infty \rbrace$, we define the full subcategory
	\[\mathcal{K}_{(p)} := \lbrace a \in \mathcal{K} : \dim(\supp(a)) \leq p \rbrace ~.\]
	We denote by $\Spc(\mathcal{K})_p$ the set of points $\mathcal{Q}$ of $\Spc(\mathcal{K})$ such that $\dim(\mathcal{Q}) = p$.
	\label{dimfuncdef}
\end{dfn}

\begin{rem}
	From the properties of $\supp(-)$, it follows that $\mathcal{K}_{(p)}$ is a thick tensor ideal in~$\mathcal{K}$. 
\end{rem}

\begin{ex}
	The main examples of dimension functions we will consider are the Krull dimension and the opposite of the Krull co-dimension. For $\mathcal{P} \in \Spc(\mathcal{K})$, its \emph{Krull dimension} $\dim_{\mathrm{Krull}}(\mathcal{P})$ is the maximal length $n$ of a chain of irreducible closed subsets
	\[\emptyset \subsetneq C_0 \subsetneq C_1 \subsetneq \ldots \subsetneq C_n = \overline{\lbrace \mathcal{P} \rbrace}. \]
	Dually, we define the \emph{opposite of the Krull co-dimension} 
	\[-\codim_{\mathrm{Krull}}(\mathcal{P})\] 
	as follows: if we have a chain of irreducible closed subsets of maximal length
	\[ \overline{\lbrace \mathcal{P} \rbrace} = C_0 \subsetneq C_1 \ldots \subsetneq C_n = \text{ maximal irred.~comp.~of $\Spc(\mathcal{K})$ containing $\mathcal{P}$}\]
	we set 
	\[-\codim_{\mathrm{Krull}}(\mathcal{P}) = -n~.\]
	\label{dimfuncex}
\end{ex}

A dimension function determines a filtration of $\mathcal{K}$. We have a chain of $\otimes$-ideals 
\[\mathcal{K}_{(-\infty)} \subset \cdots \subset \mathcal{K}_{(p)} \subset \mathcal{K}_{(p+1)} \subset \cdots \subset \mathcal{K}_{(\infty)} = \mathcal{K}~.\]
The sub-quotients of this filtration have a local description as we will see next. First let us introduce another useful property of tensor triangulated categories.
\begin{dfn}[see \cite{balmer2010tensor}*{Definition 20}]
	A tensor triangulated category $\mathcal{K}$ is called \emph{rigid} if there is an exact functor $D: \mathcal{K}^{\mathrm{op}} \to \mathcal{K}$ and a natural isomorphism $\Hom_{\mathcal{K}}(a \otimes b, c) \cong \Hom_{\mathcal{K}}(b, D(a) \otimes c)$ for all objects $a,b,c \in \mathcal{K}$. The object $D(a)$ is called the \emph{dual} of $a$.
	\label{dfnrigidttcat}
\end{dfn}

\begin{rem}
	From the natural isomorphism
	\[\Hom_{\mathcal{K}}(a \otimes b, c) \cong \Hom_{\mathcal{K}}(b, D(a) \otimes c)\]
	of Definition \ref{dfnrigidttcat}, it follows that  $a \otimes -$ and $D(a) \otimes -$ form an adjoint pair of functors for all objects $a \in \mathcal{K}$.
	\label{remrigidadjoint}
\end{rem}

We now fix a dimension function on a rigid tensor triangulated category $\mathcal{K}$ and look at the sub-quotients of the induced filtration.
\begin{thm}[see \cite{balmerfiltrations}*{Theorem 3.24}] 
	Let $\mathcal{K}$ be a rigid tensor triangulated category equipped with a dimension function $\dim$ such that $\Spc(\mathcal{K})$ is a noetherian topological space. Then, for all $p \in \mathbb{Z}$, there is an exact equivalence
	\[\left(\mathcal{K}_{(p)}/\mathcal{K}_{(p-1)}\right)^{\natural} \to \coprod_{\substack{P \in \Spc(\mathcal{K}) \\ \dim(P) = p}} \Min(\mathcal{K}_P)~.\]
	where $\mathcal{K}_P:= (\mathcal{K}/P)^{\natural}$ and $\Min(\mathcal{K}_P)$ denotes the full triangulated subcategory of objects with support the unique closed point of $\Spc(\mathcal{K}_P)$.
	\label{thmfiltdecomp}
\end{thm}

\begin{rem}
	The exact equivalence of Theorem \ref{thmfiltdecomp} is induced by the functor.
	\begin{align*}
	\mathcal{K}_{(p)}/\mathcal{K}_{(p-1)} &\to \coprod_{\substack{\mathcal{P} \in \Spc(\mathcal{K}) \\ \dim(\mathcal{P}) = p}} \Min(\mathcal{K}/\mathcal{P})\\
	a &\mapsto (Q_{\mathcal{P}}(a)) 
	\end{align*}
	where $Q_{\mathcal{P}}$ is the localization functor $\mathcal{K} \to \mathcal{K}/\mathcal{P}$. It is shown in \cite{balmerfiltrations} that the image of this functor is dense, so it induces an equivalence after idempotent completion on both sides.
\end{rem}

\section{Tensor triangular Chow groups}

\subsection{The definition}

Let us start with a definition of tensor triangular cycle groups and Chow groups right away, following the ideas from \cite{balmerchow}.
\begin{dfn}
	Let $\mathcal{K}$ be a tensor triangulated category as in Definition \ref{defttcat}, equipped with a dimension function. For $p \in \mathbb{Z}$ we define the \emph{$p$-dimensional cycle group of $\mathcal{K}$} as
	\[\Cyc^{\Delta}_{p}(\mathcal{K}) := \Kzero\left( (\mathcal{K}_{(p)}/\mathcal{K}_{(p-1)} )^{\natural}\right)~,\]
	where $\Kzero\left( (\mathcal{K}_{(p)}/\mathcal{K}_{(p-1)} )^{\natural}\right)$ is the Grothendieck group of the Ver\-dier quotient $(\mathcal{K}_{(p)}/\mathcal{K}_{(p-1)})^{\natural}$ and $\mathcal{K}_{(l)} \subset \mathcal{K}$ denotes the full triangulated subcategory of objects with dimension of support $\leq l$ (see Definition \ref{dimfuncdef}), for $l=p,p-1$.
	\label{defcycle}
\end{dfn}
We also need a generalized notion of rational equivalence, which we describe next. Look at the following diagram of subcategories and sub-quotients of~$\mathcal{K}$
\[\xymatrix{
	\mathcal{K}_{(p)} \ar@{^{(}->}[r]^I \ar@{->>}[d]^Q& \mathcal{K}_{(p+1)} \\
	\mathcal{K}_{(p)}/\mathcal{K}_{(p-1)} \ar@{^{(}->}[r]^J & (\mathcal{K}_{(p)}/\mathcal{K}_{(p-1)})^{\natural}
}
\]
where $I,J$ denote the obvious embeddings and $Q$ is the Verdier quotient functor. After applying $\Kzero$ we get a diagram
\[
\xymatrix{
	\Kzero(\mathcal{K}_{(p)}) \ar[r]^i \ar@{->>}[d]^q& \Kzero(\mathcal{K}_{(p+1)}) \\
	\Kzero(\mathcal{K}_{(p)}/\mathcal{K}_{(p-1)}) \ar@{^{(}->}[r]^(0.38){j} & \Kzero\left( (\mathcal{K}_{(p)}/\mathcal{K}_{(p-1)})^{\natural} \right) = \Cyc^{\Delta}_{p}(\mathcal{K})
}
\]
where the lowercase maps are induced by the uppercase functors.

\begin{dfn}
	Let $\mathcal{K}$ be a tensor triangulated category as in Definition \ref{defttcat}, equipped with a dimension function. For $p \in \mathbb{Z}$ we define the \emph{$p$-dimensional Chow group of $\mathcal{K}$} as
	\[\Chow^{\Delta}_{p}(\mathcal{K}) := \Cyc^{\Delta}_{p}(\mathcal{K}) / j \circ q (\ker(i)).\]
	\label{defchow}
\end{dfn}

\begin{rem}
	Let us shed some light on the motivation behind Definitions \ref{defcycle} and \ref{defchow}. The following discussion will be recalled in more detail at the beginning of Section \ref{sectionagreement}. If $X$ is an algebraic variety, then the abelian category $\mathrm{Coh}(X)$ of coherent sheaves on $X$ admits a filtration by Serre subcategories
	\[\cdots \subset M^i \subset M^{i-1} \subset \cdots \subset M^0 = \mathrm{Coh}(X)\]
	where $M^i$ is the full subcategory of $\mathrm{Coh}(X)$ consisting of those coherent sheaves with codimension of support $\geq i$. Using Quillen's d\'evissage theorem, one computes that for all $s \geq 0$,
	\[\Kzero(M^{s}/M^{s+1}) \cong \coprod_{x \in X^{(s)}} \Kzero(k(x)) \cong \coprod_{x \in X^{(s)}} \mathbb{Z} = \Cyc^{s}(X) ~.\]
	Furthermore, we have maps $i: \Kzero(M^s) \to \Kzero(M^{s-1})$ and $p: \Kzero(M^s) \to \Kzero(M^{s}/M^{s+1}) \cong \Cyc^{s}(X)$ induced by the inclusion and quotient functor, respectively. Quillen proves in \cite{quillenhigher} that $p(\ker(i))$ (in the guise of the image of a boundary map of the coninveau spectral sequence) is exactly the subgroup of cycles rationally equivalent to zero (cf.\ Theorem \ref{quillenrat}), and hence $\Kzero(M^{s}/M^{s+1})/p(\ker(i)) \cong \Chow^s(X)$. From this point of view, Definitions \ref{defcycle} and \ref{defchow} can be seen as a tensor triangulated analogon of this description of~$\Chow^s(X)$: we replace $\mathrm{Coh}(X)$ by a tensor triangulated category $\mathcal{K}$ and the subcategories $M^i$ by the subcategories $\mathcal{K}_{(i)}$, which are also defined using the (co)dimension of support of the objects of $\mathcal{K}$. Since $\Kzero(\mathcal{A}) = \Kzero(\Db(\mathcal{A}))$ for any abelian category $\mathcal{A}$, one hopes that the tensor triangular Chow groups of $\Dperf(X)$ recover the usual Chow groups of $X$ as well when $\Dperf(X) = \Db(\mathrm{Coh}(X))$, i.e.\ when $X$ is non-singular. As we will see in Theorem \ref{agreementthm}, this hope is justified.
\end{rem}

\begin{rem}
	Let us point out another analogy between the tensor triangular cycle groups $\Cyc^{\Delta}_p(\mathcal{K})$ and the usual cycle groups of an algebraic variety, which also explains the presence of the idempotent completion in Definition \ref{defcycle}: assume that $\mathcal{K}$ is a tensor triangulated category that is rigid, equipped with a dimension function and such that $\Spc(\mathcal{K})$ is a noetherian topological space. By Theorem \ref{thmfiltdecomp} the quotient functors $Q_P: \mathcal{K} \to \mathcal{K}/P$ for $P \in \Spc(\mathcal{K})$ induce an exact equivalence
	\begin{equation}
	\left(\mathcal{K}_{(p)}/\mathcal{K}_{(p-1)}\right)^{\natural} \to \coprod_{\substack{P \in \Spc(\mathcal{K}) \\ \dim(P) = p}} \Min(\mathcal{K}_P)~.
	\label{eqquotdecomp}
	\end{equation}
	The decomposition (\ref{eqquotdecomp}) is in the main reason why we idempotent-complete the Verdier quotient $\mathcal{K}_{(p)}/\mathcal{K}_{(p-1)}$. Just as in the theory of algebraic cycles, an element of the $p$-dimensional tensor triangular cycle group of $\mathcal{K}$
	\[\Cyc^{\Delta}_{p}(\mathcal{K}) = \Kzero \left( (\mathcal{K}_{(p)}/\mathcal{K}_{(p-1)})^{\natural} \right) \cong \coprod_{\substack{P \in \Spc(\mathcal{K}) \\ \dim(P) = p}} \Kzero \left(\text{Min}(\mathcal{K}_P\right))\]
	can thus be regarded as a sum of $p$-dimensional (relative to the dimension function) irreducible closed subsets of $\Spc(\mathcal{K})$ with coefficients in~$\Kzero \left(\text{Min}(\mathcal{K}_P\right))$. 
	\label{altcycrem}
\end{rem}

In the case that $\mathcal{K} = \Dperf(X)$ for $X$ a non-singular noetherian scheme, we show next that the Grothendieck group $\Kzero \left(\text{Min}(\mathcal{K}_P\right))$ group is isomorphic to $\mathbb{Z}$. This will allows us to conclude that Definition \ref{defcycle} recovers the usual cycle groups of $X$ for $\mathcal{K} = \Dperf(X)$ (see Corollary \ref{corcycrecover}). Let us first recall the following two well-known auxiliary lemmas.
\begin{lma}
	Let $R$ be a commutative local noetherian ring with maximal ideal $\mathfrak{m}$ and $M$ be a finitely generated $R$-module. Then 
	\[\supp(M) = \lbrace \mathfrak{m} \rbrace \Leftrightarrow \length(M) < \infty\]
	\label{lmasupportlength}
	\qed
\end{lma}

\begin{lma}
	Let $R$ be a commutative local ring with maximal ideal $\mathfrak{m}$ and denote by $R\!\mathrm{-fl}$ the abelian category of finite length $R$-modules. Then the map 
	\begin{align*}
	R\!\mathrm{-fl} &\to \mathbb{Z}\\
	M &\mapsto \mathrm{length}(M)
	\end{align*}
	induces an isomoprhism
	\[\Kzero(R\!\mathrm{-fl}) \cong \mathbb{Z}~.\]
	\label{lmalengthk0}
	\qed
\end{lma}

\begin{thm}
	Let $\mathcal{K} = \Dperf(X)$ for $X$ a non-singular noetherian scheme. Then 
	\[\Kzero \left(\mathrm{Min}(\mathcal{K}_P)\right) \cong \mathbb{Z}\]
	for all $P \in \Spc(\mathcal{K}) \cong X$ (see Example \ref{exschemereconstruct}).
	\label{propk0loccatintegers}
\end{thm}
\begin{proof}
	Let $\rho$ denote the isomorphism $\Spc(\Dperf(X)) \to X$, inverse to the map $\sigma$ from Example \ref{exschemereconstruct}. By \cite{balmerfiltrations}*{Section 4.1}, the category $\text{Min}(\mathcal{K}_P)$ is equivalent to
	\[\mathrm{K}^{\mathrm{b}}_{\mathrm{fin.lg.}}(\mathcal{O}_{X,\rho(P)}\mathrm{-free})~,\]
	the bounded homotopy category of complexes of free $\mathcal{O}_{X,\rho(P)}$-modules of finite rank with finite length homology. As $\mathcal{O}_{X,\rho(P)}$ is regular by assumption, every bounded complex of finitely generated $\mathcal{O}_{X,\rho(P)}$-modules is quasi-isomorphic to a bounded complex of free $\mathcal{O}_{X,\rho(P)}$-modules of finite rank. Therefore, if $\Db_{\mathrm{fin.lg.}}(\mathcal{O}_{X,\rho(P)}\Mod)$ denotes the bounded derived category of complexes of finitely generated $\mathcal{O}_{X,\rho(P)}$-modules with finite length homology, the natural functor
	\[\mathrm{K}^{\mathrm{b}}_{\mathrm{fin.lg.}}(\mathcal{O}_{X,\rho(P)}\mathrm{-free}) \to \Db_{\mathrm{fin.lg.}}(\mathcal{O}_{X,\rho(P)}\Mod)\]
	gives rise to an equivalence of categories
	\[\mathrm{K}^{\mathrm{b}}_{\mathrm{fin.lg.}}(\mathcal{O}_{X,\rho(P)}\mathrm{-free}) \cong \Db_{\mathrm{fin.lg.}}(\mathcal{O}_{X,\rho(P)}\Mod).\] 
	The latter category is in turn equivalent to $\Db(\mathcal{O}_{X,\rho(P)}\!\mathrm{-fl})$, the bounded derived category of finite length modules over $\mathcal{O}_{X,\rho(P)}$. Indeed, by Lemma \ref{lmasupportlength}, for a finitely generated module $M$ over $\mathcal{O}_{X,\rho(P)}$, having finite length is the same as being supported on the unique closed point $P_0$ of $\mathrm{Spec}(\mathcal{O}_{X,\rho(P)})$. 
	The result then follows by \cite{keller}*{Section 1.15, Example b)}, where it is shown that for a commutative noetherian ring $R$, and $\mathcal{A} \subset R\Mod$ the full abelian subcategory of finitely generated $R$-modules supported on a closed subscheme $Z$ of $\mathrm{Spec}(R)$, there is an equivalence of categories 
	\[\Db(\mathcal{A}) \cong \mathrm{D}^{\mathrm{b}}_{\mathcal{A}}(R\Mod)~,\] 
	where the latter expression denotes the full subcategory of $\Db(R\Mod)$ consisting of complexes with homology in $\mathcal{A}$. 
	
	Summarizing, we have
	\begin{equation}
	\Kzero \left(\text{Min}(\mathcal{K}_P\right)) \cong  \Kzero \left(\Db(\mathcal{O}_{X,\rho(P)}\!\mathrm{-fl})\right) \cong \Kzero \left(\mathcal{O}_{X,\rho(P)}\!\mathrm{-fl}\right) \cong \mathbb{Z}
	\label{eqExplicitMinIso}
	\end{equation}
	where and the last isomorphism is induced by the length function as in Lemma \ref{lmalengthk0}. 
\end{proof}

\begin{rem}
	We can make the isomorphism $\Kzero \left(\mathrm{Min}(\mathcal{K}_P)\right) \cong \mathbb{Z}$ from Theorem \ref{propk0loccatintegers} explicit if we identify $\Kzero \left(\text{Min}(\mathcal{K}_P\right))$ with $\mathrm{K}^{\mathrm{b}}_{\mathrm{fin.lg.}}(\mathcal{O}_{X,\rho(P)}\mathrm{-free})$ as in the proof of Theorem \ref{propk0loccatintegers}. We compose the isomorphism 
	\begin{align*}
	\Kzero \left(\Db(\mathcal{O}_{X,\rho(P)}\!\mathrm{-fl})\right) &\to \Kzero \left(\mathcal{O}_{X,\rho(P)}\!\mathrm{-fl}\right)\\
	[A_{\bullet}] &\mapsto \sum_i (-1)^i [\mathrm{H}^i(A_{\bullet})]
	\end{align*}
	and the length function as in (\ref{eqExplicitMinIso}) to obtain the following: if $A_{\bullet}$ is a complex in $\mathrm{K}^{\mathrm{b}}_{\mathrm{fin.lg.}}(\mathcal{O}_{X,\rho(P)}\mathrm{-free})$, then the map is given as
	\[\theta_{\rho(P)}: [A_{\bullet}] \mapsto \sum_i (-1)^i \mathrm{length}(\mathrm{H}^i(A_{\bullet}))~.\]
	\label{remexplicitlocaliso}
\end{rem}

\begin{cor}
	Let $X$ be a non-singular noetherian scheme. Let $\mathcal{K} = \Dperf(X)$ be equipped with the opposite of the Krull codimension as a dimension function (see Example \ref{dimfuncex}). Then the map
	\[\theta_{-p} := \coprod_{\substack{\rho(P) \\ \dim(\rho(P)) = -p}} \theta_{\rho(P)}\]
	with $\theta_{\rho(P)}$ as in Remark \ref{remexplicitlocaliso} induces an isomorphism \[\Cyc^{\Delta}_{-p}(\mathcal{K}) \cong \Cyc^p(X)\]
	for all $p \geq 0$.
	\label{corcycrecover}
\end{cor}
\begin{proof}
	Let $p \geq 0$. Using Remark \ref{altcycrem} and the isomorphism $\Spc(\mathcal{K}) \cong X$ we have a chain of isomorphisms 
	\begin{align*}
	\Cyc^{\Delta}_{-p}(\mathcal{K}) &\cong \coprod_{P \in \text{Spc}(\mathcal{K})_{-p}} \Kzero \left(\text{Min}(\mathcal{K}_P\right)) \\
	&\cong \coprod_{\rho(P) \in X^{(p)}} \Kzero\left(\mathrm{K}^{\mathrm{b}}_{\mathrm{fin.lg.}}(\mathcal{O}_{X,\rho(P)}\mathrm{-free})\right) \\ 
	&\cong \coprod_{P \in X^{(p)}} \mathbb{Z} \\
	&\cong  \Cyc^p(X)~,
	\end{align*}
	where the penultimate map is given by $\theta_{-p}$.
\end{proof}

\subsection{Agreement with algebraic geometry}
\label{sectionagreement}
We want to show now that the tensor triangular Chow groups carry their name for a reason. As we will see, they are --- at least in the non-singular case --- not only defined analogously to, but \emph{are} an honest generalization of the classical Chow groups from algebraic geometry. 
\begin{conv}
	We now fix some notation for the rest of the section: if not explicitly stated otherwise, let $X$ denote a separated, non-singular scheme of finite type over a field $k$, and $\Dperf(X)$ be the derived category of perfect complexes of $\mathcal{O}_X$-modules, which is equivalent to $\Db(X)$, the bounded derived category of coherent sheaves on~$X$. We will also assume that $\Dperf(X)$ is equipped with $-\codim_{\mathrm{Krull}}$ as a dimension function.
	\label{convschemeass}
\end{conv}

In order to proceed, we will need some higher algebraic K-theory as developed by Quillen. We recall the following material from \cite{quillenhigher}*{\S7}: consider the abelian category $\mathrm{Coh}(X)$ of coherent sheaves on~$X$. There is a filtration of this category by codimension of support:
\[\cdots \subset M^i \subset M^{i-1} \subset \cdots \subset M^0 = \mathrm{Coh}(X) \]
where $M^p$ denotes the subcategory of coherent sheaves whose codimension of support is $\geq p$. The subcategory $M^p \subset M^{p+1}$ is a \emph{Serre subcategory}, i.e.\ a full subcategory such that if 
\[0 \to A \to B \to C \to 0\] 
is an exact sequence in $M^{p+1}$, then $A,C$ are objects of $M^p$ if and only if $B$ is one. This property allows us to define the quotient abelian category $M^{p+1}/M^{p}$ and thus, for every $p$, there is an exact sequence of abelian categories
\[M^{p+1} \hookrightarrow M^p \twoheadrightarrow M^p/M^{p+1}\]
which induces a long exact localization sequence of K-groups
\begin{equation}
\begin{gathered}
\xymatrix{
	\cdots\ar[r] & \mathrm{K}_j(M^{p+1}) \ar[r]^{i^p_j} &\mathrm{K}_j(M^{p}) 
	\ar[r]^-{q^p_j} & \mathrm{K}_j(M^{p}/M^{p+1}) \ar `r[d]`[l]`[llld]^{b^p_j}`[d][lld] \\
	& \mathrm{K}_{j-1}(M^{p+1}) \ar[r]^-{i^p_{j-1}} &\mathrm{K}_{j-1}(M^{p}) 
	\ar[r]^-{q^p_{j-1}} & \mathrm{K}_{j-1}(M^{p}/M^{p+1}) \ar[r] & \cdots
}
\end{gathered}
\label{localisationseq}
\end{equation}
Combining these long exact sequences for all $p$, we can form the associated exact couple and obtain the Quillen coniveau spectral sequence as in \cite{quillenhigher}*{\S7, Theorem 5.4} with $E_1$-page
\[E^{p,q}_1 = \mathrm{K}_{-p-q}(M^{p}/M^{p+1}).\]
We are especially interested in the boundary map 
\begin{equation}
d_1: \mathrm{K}_1(M^{s-1}/M^{s}) \overset{b^{s-1}_1}{\longrightarrow} \Kzero(M^s) \overset{q^s_0}{\longrightarrow} \Kzero(M^{s}/M^{s+1})
\label{localisationboundary}
\end{equation}
of this spectral sequence. Using that
\begin{equation}
\mathrm{K}_{i}(M^{s}/M^{s+1}) \cong \coprod_{x \in X^{(s)}} \mathrm{K}_i(k(x)) ~,
\label{eqdevissageiso}
\end{equation}
where $X^{(s)}$ denotes the set of points of $X$ whose closure has codimension $s$ in $X$, Quillen proves the following:
\begin{thm}[cf. \cite{quillenhigher}*{\S7, Proposition 5.14}]
	The image of 
	\[d_1: \mathrm{K}_1(M^{s-1}/M^{s}) \longrightarrow \Kzero(M^{s}/M^{s+1}) \cong \coprod_{x \in X^{(s)}} \Kzero(k(x)) \cong \coprod_{x \in X^{(s)}} \mathbb{Z} = \Cyc^{s}(X)\]
	is the subgroup of codimension-$p$ cycles rationally equivalent to zero. In other words, there is an isomorphism~$\coker(d_1) \cong \Chow^s(X)$.
	\label{quillenrat}
\end{thm}

\begin{rem}
The regularity assumption on $X$ is not necessary for Theorem \ref{quillenrat}.
\end{rem}

In our setting, we work with the triangulated category $\Dperf(X) \cong \Db(X)$ instead of the abelian category~$\mathrm{Coh}(X)$. Recall that the defining diagram for the tensor triangular Chow groups in this case is given as follows:
\[
\xymatrix{
	\Kzero(\Db(X)_{(p)}) \ar[r]^i \ar@{->>}[d]^q& \Kzero(\Db(X)_{(p+1)}) \\
	\Kzero(\Db(X)_{(p)}/\Db(X)_{(p-1)}) \ar@{^{(}->}[d]^(0.4){j} &  \\
	\underset{\textstyle = \Cyc^{\Delta}_{p}(\mathcal{K})}{\underbrace{\Kzero\left( (\Db(X)_{(p)}/\Db(X)_{(p-1)})^{\natural} \right)}}&
}
\]
This diagram maps to a similar one involving the related abelian categories:
\begin{equation}\label{agreementdiag}
\begin{gathered}
\xymatrix{
	\Kzero\left(\Db(X)_{(p)}\right) \ar[r]^i \ar@{->>}[d]^q \ar[dr]& \Kzero\left(\Db(X)_{(p+1)}\right) \ar[dr]&\\
	\Kzero\left(\Db(X)_{(p)}/\Db(X)_{(p-1)}\right) \ar@{^{(}->}[d]^(0.4){j} \ar[dr]&  \Kzero\left(M^{-p}\right) \ar[d]^{q_0} \ar[r]^{i_0}& \Kzero\left(M^{-p-1}\right)\\
	\underset{\textstyle = \Cyc^{\Delta}_{p}(\mathcal{K})}{\underbrace{\Kzero\left((\Db(X)_{(p)}/\Db(X)_{(p-1)})^{\natural} \right)}} & \Kzero\left(M^{-p}/M^{-p+1}\right) &
}
\end{gathered}
\end{equation}
The diagonal homomorphisms are all given by the formula
\begin{equation}
[C^{\bullet}] \mapsto \sum_i (-1)^i [H^i(C^{\bullet})]~.
\label{eqk0derivediso}
\end{equation}
We proceed to show that these are actually all isomorphisms, which follows from the fact that there are exact equivalences
\begin{equation}
\Db(X)_{(q)} \cong \Db(M^{-q})
\label{equiv1}
\end{equation}
and
\begin{equation}
\Db(X)_{(q)}/\Db(X)_{(q-1)} \cong \Db(M^{-q}/M^{-q+1})
\label{equiv2}
\end{equation}
for all $q \in \mathbb{Z}$. Indeed, the diagonal maps are then just the usual isomorphisms between $\Kzero(\Db(\mathcal{A}))$ and $\Kzero(\mathcal{A})$ for some abelian category $\mathcal{A}$. This also proves that $j$ is the identity morphism, as the derived category of an abelian category is idempotent complete \cite{balschlichidem}*{Corollary 2.10}.

The proof of the equivalences (\ref{equiv1}) and (\ref{equiv2}) is a consequence of the following theorem:
\begin{thm}[see \cite{keller}*{Section 1.15}]
	Let $\mathcal{B}$ be an abelian category and $\mathcal{A} \subset \mathcal{B}$ a Serre subcategory with quotient $\mathcal{B}/\mathcal{A}$. Assume that the following criterion holds: for each exact sequence
	\[0 \to A \to B \to C \to 0\]
	in $\mathcal{B}$ with $A \in \mathcal{A}$, there is a commutative diagram with exact rows
	\[
	\xymatrix{
		0 \ar[r]& A \ar[r] \ar[d]^{\id}& B \ar[r] \ar[d]^f& C \ar[r] \ar[d]^g& 0 \\
		0 \ar[r]& A \ar[r] & A' \ar[r] & A'' \ar[r] & 0
	}
	\]
	such that $A',A''$ are objects of $\mathcal{A}$. 
	
	Then, there is an exact equivalence of triangulated categories induced by the inclusion
	\[\Db(\mathcal{A}) \to \mathrm{D}^{\mathrm{b}}_{\mathcal{A}}(\mathcal{B})~,\]
	where $\mathrm{D}^{\mathrm{b}}_{\mathcal{A}}(\mathcal{B}) \subset \Db(\mathcal{B})$ denotes the full subcategory of complexes with homology in~$\mathcal{A}$. Furthermore, in the induced sequence of triangulated categories
	\[\Db(\mathcal{A}) \xrightarrow{i} \Db(\mathcal{B}) \xrightarrow{q} \Db(\mathcal{B}/\mathcal{A})~,\]
	the functor $i$ is fully faithful and $\Db(\mathcal{B}/\mathcal{A}) \cong \Db(\mathcal{B})/\Db(\mathcal{A})$ (i.e.\ the sequence is \emph{exact}).
	\label{thmkellercond}
\end{thm}

Let us verify that the conditions for Theorem \ref{thmkellercond} are satisfied in our case.
\begin{lma}
	Let $0 \to A \to B \to C \to 0$ be an exact sequence in~$\mathrm{Coh}(X)$. Then there exist coherent sheaves $A',A''$ on $X$ with $\supp(A'),\supp(A'') \subset \supp(A)$ that fit into a commutative diagram with exact rows
	\[
	\xymatrix{
		0 \ar[r]& A \ar[r] \ar[d]^{\id}& B \ar[r] \ar[d]^f& C \ar[r] \ar[d]^g& 0 \\
		0 \ar[r]& A \ar[r] & A' \ar[r] & A'' \ar[r] & 0
	}
	\]
	\label{kellercondition}
\end{lma}
\begin{proof}
	Suppose that $A$ is supported on a closed subscheme with associated ideal sheaf~$I$. As $X$ is noetherian, we can use the sheaf-theoretic version of the Artin-Rees lemma (cf.\ \cite{stacks-project}*{Lemma 29.10.3}) which says that there exists a $c > 0$ such that for all $n > c$ we have~$I^n B \cap A = I^{n-c}(I^c B \cap A)$. Now take some $n_0$ such that $I^{n_0-c}A = 0$, then we get the diagram
	\[
	\xymatrix{
		0 \ar[r]& A \ar[r] \ar[d]^{\id}& B \ar[r] \ar[d]& C \ar[r] \ar[d]& 0 \\
		0 \ar[r]& A \ar[r] & B/(I^{n_0} B) \ar[r] & C/(I^{n_0} C) \ar[r] & 0
	}
	\]
	where the vertical arrows are given by the canonical projections. It is easy to see that the diagram commutes and that all sheaves in the lower row have their support contained in~$\supp(A)$.
\end{proof}
Since we have checked the conditions of Theorem \ref{thmkellercond}, its first statement tells us that the equivalence (\ref{equiv1}) holds, because we can write $\Db(X)_{(q)}$ as $\mathrm{D}^{\mathrm{b}}_{M^{-q}}(\mathrm{Coh}(X))$. The equivalence (\ref{equiv2}) holds by the second statement of Theorem \ref{thmkellercond}, which says that
\[\Db(M^{-q}/M^{-q+1}) \cong \Db(M^{-q})/\Db(M^{-q+1})~,\]
where the latter expression is equivalent to $\Db(X)_{(q)}/\Db(X)_{(q-1)}$ by the first part of Theorem \ref{thmkellercond}.

As we know that the diagonal maps in diagram (\ref{agreementdiag}) are isomorphisms and that $j$ is the identity morphism we see that 
\[j \circ q (\ker(i)) \cong q_0(\ker(i_0)) = q_0(\im(b_0)) = \im(d_1)\] 
(see Theorem \ref{quillenrat}). We have thus proved the following:
\begin{thm}
	Let $X$ be a separated, non-singular scheme of finite type over a field and assume that the tensor triangulated category $\Dperf(X)$ is equipped with the dimension function~$-\codim_{\mathrm{Krull}}$. Then there are isomorphisms
	\[\Cyc^{\Delta}_{p}\left(\Dperf(X)\right) \cong Z^{-p}(X) \qquad \text{and} \qquad \Chow^{\Delta}_{p}\left(\Dperf(X)\right) \cong \Chow^{-p}(X) \]
	for all $p \in \mathbb{Z}$.
	\qed
	\label{agreementthm}
\end{thm}

A couple of remarks are in order:
\begin{rem}
	The proof of  Theorem \ref{agreementthm} shows that even if $X$ is singular, we have isomorphisms
	\[\Kzero\left(\mathrm{D^b}(X)_{(p)}/\mathrm{D^b}(X)_{(p-1)}\right) \cong \Cyc^{-p}(X)\]
	and
	\[\Kzero\left(\mathrm{D^b}(X)_{(p)}/\mathrm{D^b}(X)_{(p-1)}\right)/q(\ker(i)) \cong \Chow^{-p}(X)\]
	for all $p$. Furthermore, the inclusion functors $\Dperf(X)_{(p)} \to \mathrm{D^b}(X)_{(p)}$ always induce exact functors 
	\[\Dperf(X)_{(p)}/\Dperf(X)_{(p-1)} \to \Db(X)_{(p)}/\Db(X)_{(p-1)}~.\] 
	As $\Db(X)_{(p)}/\Db(X)_{(p-1)}$ is idempotent complete, we obtain exact  functors
	\[\left(\Dperf(X)_{(p)}/\Dperf(X)_{(p-1)}\right)^{\natural} \to \Db(X)_{(p)}/\Db(X)_{(p-1)}\]
	and after applying $\Kzero$ we get induced homomorphisms
	\[\Cyc_p\left(\Dperf(X)\right) \to \Kzero\left(\mathrm{D^b}(X)_{(p)}/\mathrm{D^b}(X)_{(p-1)}\right) \cong \Cyc^{-p}(X)\]
	for all $p$. Similarly, we obtain homomorphisms
	\[\Chow_p\left(\Dperf(X)\right) \to  \Chow^{-p}(X)\]
	for all $p$.
\end{rem}

\begin{rem}
	Let us sketch the argument for a more ``high-level'' proof of Theorem \ref{agreementthm} using Waldhausen models for the categories~$\Dperf(X)_{(p)}$: for $p\in \mathbb{Z}$, we consider the category $\mathrm{Perf}_{(p)}(X)$ of perfect complexes on $X$ with codimension of homological support ~$\geq -p$. This category is a \emph{Waldhausen category}, i.e.\ a category with two classes of morphisms called the \emph{cofibrations} and the \emph{weak equivalences}, which both have to satisfy a list of axioms (see \cite{waldktheory}). For a Waldhausen category $W$, we can define higher algebraic $\mathrm{K}$-groups $\mathrm{K}_i(W)$ for $i \geq 0$ as in \cite{waldktheory}. For $\mathrm{Perf}_{(p)}(X)$, the cofibrations are given by the degree-wise split monomorphisms of complexes, and the weak equivalences are given as the quasi-isomorphisms. If we define the Waldhausen category $\mathrm{Perf}_{(p/p-1)}(X)$ as the category $\mathrm{Perf}_{(p)}(X)$ with the same class of cofibrations but with the weak equivalences those morphisms whose mapping cone is quasi-isomorphic to an object of $\mathrm{Perf}_{(p-1)}(X)$, we obtain a sequence of Waldhausen categories
	\[\mathrm{Perf}_{(p-1)}(X) \to \mathrm{Perf}_{(p)}(X) \to \mathrm{Perf}_{(p/p-1)}(X)\]
	where both functors are given by inclusion. From the localization theorem of \cite{thomason-trobaugh}*{Theorem 1.8.2}, we obtain a long exact localization sequence
	\begin{equation}
	\begin{gathered}
	\xymatrix{
		\cdots\ar[r] & \mathrm{K}_j(\mathrm{Perf}_{(p-1)}(X)) \ar[r] &\mathrm{K}_j(\mathrm{Perf}_{(p)}(X)) 
		\ar[r] & \mathrm{K}_j(\mathrm{Perf}_{(p/p-1)}(X)) \ar `r[d]`[l]`[llld]`[d][lld] \\
		& \mathrm{K}_{j-1}(\mathrm{Perf}_{(p-1)}(X)) \ar[r] &\mathrm{K}_{j-1}(\mathrm{Perf}_{(p)}(X)) 
		\ar[r] & \cdots &
	}
	\end{gathered}
	\label{eqwaldlocseq}
	\end{equation}
	By the regularity of $X$, $\mathrm{Perf}_{p-1}(X)$ and $\mathrm{Perf}_{p}(X)$ coincide with $\mathrm{C}^b(\mathrm{Coh}(X))_{(p-1)}$ and $\mathrm{C}^b(\mathrm{Coh}(X))_{(p)}$, the categories of bounded complexes of coherent sheaves on $X$ with codimension of homological support $\geq -p+1$ and $\geq -p$, respectively. Their Waldhausen $\mathrm{K}$-theory is in turn isomorphic to the Waldhausen $\mathrm{K}$-theory of $\mathrm{C}^b(\mathrm{Coh}(X)^{(-p+1)})$ and $\mathrm{C}^b(\mathrm{Coh}(X)^{(-p)})$ respectively by \cite{thomason-trobaugh}*{Theorem 1.9.8}, as the natural inclusions induce equivalences on the corresponding derived categories. The natural functor 
	\[\mathrm{Perf}_{(p/p-1)}(X) \to \mathrm{C}^b(\mathrm{Coh}(X)^{(-p)}/\mathrm{Coh}(X)^{(-p+1)})\] 
	also induces an equivalence on the level of derived categories and thus we apply \cite{thomason-trobaugh}*{Theorem 1.9.8} again to obtain that the corresponding Waldhausen $\mathrm{K}$-theories of the involved categories coincide. Finally, the comparison to Quillen $\mathrm{K}$-theory of \cite{thomason-trobaugh}*{Theorem 1.11.2 and Theorem 1.11.7} yields that the sequences (\ref{eqwaldlocseq}) are isomorphic to the sequences (\ref{localisationseq}) and by forming the associated exact couple, we get a new spectral sequence which is isomorphic to Quillen's coniveau spectral sequence. In particular, we can talk about the cokernel of the map $d_1$ (as in (\ref{localisationboundary})) in this new spectral sequence which is then isomorphic to the cokernel of $d_1$ in Quillen's coniveau spectral sequence which is in turn isomorphic to $\Chow^{-p}(X)$.
\end{rem}

\begin{prop}
	The isomorphism
	\[\rho_X: \Cyc^{\Delta}_p(\Dperf(X)) \xrightarrow{\sim} \Cyc^{-p}(X)\]
	is explicitly given as follows: if $C^{\bullet}$ is an object of $\Dperf(X)_{(p)}/\Dperf(X)_{(p-1)}$, then
	\[\rho_X([C^{\bullet}]) = \sum_i \sum_{x \in X^{(-p)}} (-1)^i \length_{\mathcal{O}_{X,x}}\left(\mathrm{H}^i\left(C^{\bullet}\right)_x\right) \cdot \overline{\lbrace x \rbrace}~.\]
	\label{propagreement}
\end{prop}
\begin{proof}
	The isomorphism (\ref{eqdevissageiso})
	\[\mathrm{K}_i(M^{p}/M^{p+1}) \xrightarrow{\sim} \coprod_{x \in X^{(p)}} \mathrm{K}_i(k(x))\] 
	is explicitly given as follows: first note that we have an equivalence of categories
	\[M^{p}/M^{p+1} \xrightarrow{\sim} \coprod_{x \in X^{(p)}} \mathcal{O}_{X,x}\!\mathrm{-fl}\]
	induced by the functor 
	\begin{align*}
	M^{p} &\rightarrow \coprod_{x \in X^{(p)}} \mathcal{O}_{X,x}\!\mathrm{-fl}\\
	a &\mapsto (a_x)_{x \in X^{(p)}}
	\end{align*}
	(see e.g.\ \cite{weibelkbook}*{Chapter V, \S 9}) which in turn induces an isomorphism
	\[\mathrm{K_i}(M^{p}/M^{p+1}) \xrightarrow{\sim} \coprod_{x \in X^{(p)}} \mathrm{K}_i\left(\mathcal{O}_{X,x}\!\mathrm{-fl} \right)~.\]
	Then we have an isomorphism
	\[\coprod_{x \in X^{(p)}} \mathrm{K}_i\left(\mathcal{O}_{X,x}\!\mathrm{-fl} \right) \xleftarrow{\sim} \coprod_{x \in X^{(p)}} \mathrm{K}_i(k(x))\]
	given by componentwise \emph{d\'evissage} (see \cite{quillenhigher}), i.e.\ the inclusion of the category of finite-dimensional $k(x)$-vector spaces into $\mathcal{O}_{X,x}\!\mathrm{-fl}$ induces an isomorphism in $\mathrm{K}$-theory. For $i=0$, we have already seen this in Lemma \ref{lmalengthk0}: any element $[a_x] \in \mathrm{K}_0\left(\mathcal{O}_{X,x}\!\mathrm{-fl}\right)$ can be written as $n \cdot [k(x)]$, where $n = \length(a_x)$. We conclude that for $i=0$, the isomorphism (\ref{eqdevissageiso}) is given explicitly as
	\[[a] \mapsto \sum_{x \in X^{(p)}} \length_{\mathcal{O}_{X,x}}(a_x) \cdot \overline{\lbrace x \rbrace}\]
	Precomposing with formula (\ref{eqk0derivediso}), we obtain the explicit description
	\[\rho_X([C^{\bullet}]) = \sum_i \sum_{x \in X^{(-p)}} (-1)^i \length_{\mathcal{O}_{X,x}}\left(\mathrm{H}^i\left(C^{\bullet}\right)_x\right) \cdot \overline{\lbrace x \rbrace}~.\]
	as desired.
\end{proof}
We shall denote the induced isomorphism 
\[\Chow^{\Delta}_p(\Dperf(X)) \to \Chow^{-p}(X)\] 
by $\rho_X$ as well.

\section{Functoriality}
As we now have a reasonable definition of tensor triangular Chow groups at hand, we would like to check that it has the functoriality properties one would expect it to have from the behavior of the algebro-geometric Chow groups.

\subsection{Functors with a relative dimension}
We first have to define which class of functors we allow. In this section, $\mathcal{K}$ and $\mathcal{L}$ will always denote tensor triangulated categories as in Definition \ref{defttcat}, and we assume that both are equipped with a dimension function.
\begin{dfn}
	Let $F: \mathcal{K} \to \mathcal{L}$ be an exact functor. We say that $F$ has \emph{relative dimension $n$} if there exists some $n \in \mathbb{Z}$ such that $F(\mathcal{K}_{(p)}) \subset \mathcal{L}_{(p+n)}$ for all $p$, and $n$ is the smallest integer such that this relation holds.
	\label{funcreldimdef}
\end{dfn}
\begin{rem}
	We \emph{do not} require that $F$ is a tensor functor (cf.\ Proposition \ref{proptensorreldimzero}, Section \ref{example:properpfwd}). The composition of two functors of relative dimension $n$ and $m$ is a functor of relative dimension at most $n+m$. In almost all of the examples that follow, we have~$n=0$ (see Section \ref{subsectdervsstab} for a functor $q: \Db(kG\Mod) \to kG\stab$ of relative dimension $-1$). Note that if $F$ has relative dimension $n$, we can always arrange $n=0$ by replacing the dimension function on $\mathcal{L}$ by an appropriately shifted one.
	\label{reldimrem}
\end{rem}

\begin{thm}
	Let $F: \mathcal{K} \to \mathcal{L}$ be a functor of relative dimension $\leq n$. Then $F$ induces group homomorphisms
	\[\mathrm{z}^n_p(F): \Cyc^{\Delta}_{p}(\mathcal{K}) \to \Cyc^{\Delta}_{p+n}(\mathcal{L}) \qquad \text{and} \qquad \mathrm{c}^n_p(F): \Chow^{\Delta}_{p}(\mathcal{K}) \to \Chow^{\Delta}_{p+n}(\mathcal{L})\]
	for all $p \in \mathbb{Z}$. If $F$ has relative dimension $<n$, then $\mathrm{z}^n_p(F)$ and $\mathrm{c}^n_p(F)$ are both trivial.
	\label{functoriality}
\end{thm}
\begin{proof}
	We have the following commutative diagram
	\[
	\xymatrix{
		\mathcal{K}_{(p)} \ar@{^{(}->}[r]^{J_{\mathcal{K}}} \ar@{->>}[d]^{Q_{\mathcal{K}}} \ar[dr]^{F_p} & \mathcal{K}_{(p+1)} \ar[dr]^{F_{p+1}}& \\
		\mathcal{K}_{(p)}/\mathcal{K}_{(p-1)} \ar@{^{(}->}[d]^{I_{\mathcal{K}}} \ar[dr]^{\overline{F}} & \mathcal{L}_{(p+n)} \ar@{^{(}->}[r]^{J_{\mathcal{L}}} \ar@{->>}[d]^{Q_{\mathcal{L}}} & \mathcal{L}_{(p+n+1)} \\
		(\mathcal{K}_{(p)}/\mathcal{K}_{(p-1)})^{\natural} \ar[dr]^{\hat{F}}& \mathcal{L}_{(p+n)}/\mathcal{L}_{(p+n-1)} \ar@{^{(}->}[d]^{I_{\mathcal{L}}} & \\
		& (\mathcal{L}_{(p+n)}/\mathcal{L}_{(p+n-1)})^{\natural} &
	}
	\]
	where $F_i$ is the restriction of $F$ to $\mathcal{K}_{(i)}$ for $i=p,p+1$, $\overline{F}$ exists because 
	\[F(\mathcal{K}_{p-1}) \subset \mathcal{L}_{p+n-1} = \ker(Q_{\mathcal{L}})\] 
	and $\hat{F}$ exists as $I_{\mathcal{L}} \circ \overline{F}$ is a functor to an idempotent complete category. Applying the functor $\Kzero(-)$ yields the diagram
	\[
	\xymatrix{
		\Kzero \left(\mathcal{K}_{(p)} \right) \ar[r]^{j_{\mathcal{K}}} \ar@{->>}[d]^{q_{\mathcal{K}}} \ar[dr]^{f_p} & \Kzero \left(\mathcal{K}_{(p+1)} \right) \ar[dr]^{f_{p+1}}& \\
		\Kzero \left(\mathcal{K}_{(p)}/\mathcal{K}_{(p-1)} \right) \ar@{^{(}->}[d]^{i_{\mathcal{K}}} \ar[dr]^{\overline{f}} & \Kzero \left(\mathcal{L}_{(p+n)}\right) \ar[r]^{j_{\mathcal{L}}} \ar@{->>}[d]^{q_{\mathcal{L}}} & \Kzero \left(\mathcal{L}_{(p+n+1)} \right) \\
		\Kzero\left((\mathcal{K}_{(p)}/\mathcal{K}_{(p-1)})^{\natural}\right) \ar[dr]^{\hat{f}}& \Kzero\left(\mathcal{L}_{(p+n)}/\mathcal{L}_{(p+n-1)}\right) \ar@{^{(}->}[d]^{i_{\mathcal{L}}} & \\
		& \Kzero\left((\mathcal{L}_{(p+n)}/\mathcal{L}_{(p+n-1)})^{\natural}\right) &
	}
	\]
	where the lowercase arrows are induced by the corresponding uppercase ones. We set $\mathrm{z}^n_p(F) := \hat{f}$. From the diagram, we also deduce that 
	\[\hat{f} \circ i_{\mathcal{K}} \circ q_{\mathcal{K}} (\ker(j_{\mathcal{K}})) \subset i_{\mathcal{L}} \circ q_{\mathcal{L}} (\ker(j_{\mathcal{L}}))\]
	which implies that $\hat{f}$ also induces a homomorphism $\mathrm{c}^n_p(F)$ between the factor groups.
	
	If the relative dimension of $F$ is $m = n-r$ for some $r \geq 1$, then
	\[F(\mathcal{K}_{(p)}) \subset \mathcal{L}_{(p+m)} = \mathcal{L}_{(p+n-r)} \subset \mathcal{L}_{(p+n-1)}~.\]
	Therefore $\mathrm{z}^n_p(F)$ and $\mathrm{c}^n_p(F)$ are both $0$ in this case.
\end{proof}

\begin{notn}
	If $F: \mathcal{K} \to \mathcal{L}$ has relative dimension $m$, then we will denote the induced homomorphisms $\mathrm{z}^m_p(F)$ and $\mathrm{c}^m_p(F)$ from Theorem \ref{functoriality} by $\Cyc^{\Delta}_{p}(F)$ and $\Chow^{\Delta}_{p}(F)$, respectively.
\end{notn}

\begin{rem}
	Theorem \ref{functoriality} and Remark \ref{reldimrem} show that for all $p$, $\Cyc^{\Delta}_{p}(-)$ and $\Chow^{\Delta}_{p}(-)$ are functors from the category of essentially small tensor triangulated categories equipped with a dimension function to the category of abelian groups, with respect to the class of functors with a relative dimension.
\end{rem}

Let us finish the discussion with a general example of a functor with relative dimension~$0$.
\begin{prop}
	Let $a \in \mathcal{K}$ be an object such that $\dim(\supp(a)) \neq \pm \infty$. Then the functor 
	\[a \otimes - : \mathcal{K} \to \mathcal{K} \]
	has relative dimension~$0$.
	\label{proptensorreldimzero}
\end{prop}
\begin{proof}
	For any object $b \in \mathcal{K}$, we have 
	\[\supp(a \otimes b) = \supp(a) \cap \supp(b) \subset \supp(b)~,\]
	from which it follows that $\dim(\supp(a \otimes b)) \leq \dim(\supp(b))$. Thus $a \otimes -$ has relative dimension $\leq 0$. But $\supp(a \otimes a) = \supp(a)$ and therefore $\dim(\supp(a \otimes a)) = \dim(\supp(a))$, which shows that $a \otimes -$ leaves the dimension of support of the object $a$ fixed and finite. We conclude that $a \otimes -$ has relative dimension~$0$.
\end{proof}

\begin{rem}
	If $\dim(\supp(a)) = -\infty$, then $\supp(a \otimes b) \subset \supp(a)$ implies that $a \otimes \mathcal{K}_{(p)} \subset \mathcal{K}_{(-\infty)}$ for all $p \in \mathbb{Z}$. This shows that the functor $a \otimes -$ does not have a relative dimension in this case. If $\dim(\supp(a)) = \infty$ the situation seems to be more subtle. At least $\supp(a \otimes b) \subset \supp(b)$ implies that the relative dimension, if it exists, is bounded from above by~$0$. Here is an example where it doesn't exist: Take $\mathcal{K} = \Dperf(\mathbb{A}^1_k)$, then $\Spc(\mathcal{K}) = \mathbb{A}^1_k$ with generic point $\eta$. Pick a closed point $P \in \mathbb{A}^1_k$ and define a dimension function $\dim_P$ on $\mathbb{A}^1_k$ by setting $\dim_P(Q) = 0$ for $Q \neq P, \eta$ and $\dim_P(P) = \dim_P(\eta) = \infty$. If we look at the skyscraper sheaf $a:= k(P)$, then $\dim_P(\supp(a)) = \dim_P(P) = \infty$. Furthermore, $\mathcal{K}_{(p)}$ consists of all those bounded complexes of coherent sheaves on $\mathbb{A}^1_k$ whose homological support does not contain $P$ or $\eta$ for $p=0$ and is $\lbrace 0 \rbrace$ for $0 \neq p \in \mathbb{Z}$. Since $\supp(x \otimes y) = \supp(x) \cap \supp(y)$ for all objects $x,y \in \mathcal{K}$, we see that $a \otimes \mathcal{K}_{(p)} = \lbrace 0 \rbrace = \mathcal{K}_{(-\infty)}$ for all $p \in \mathbb{Z}$, which shows that the functor $a \otimes -$ does not admit a relative dimension in this case.
\end{rem}

\subsection{Projection formulas and relative dimension.}
For a pair of adjoint functors $(f_*,f^*)$ with relative dimensions $\dim(f_*)$ and $\dim(f^*)$, we can give a relation between $\dim(f_*)$ and $\dim(f^*)$ if they behave similarly as the derived direct image and inverse image functor in the derived projection formula from algebraic geometry (see e.g.\ \cite{huybrechts}*{p.\ 83}).

\begin{dfn}
	Let $\mathcal{C},\mathcal{D}$ be tensor triangulated categories as in Definition \ref{defttcat}, that are both equipped with a dimension function. Assume we are given an adjoint pair of exact functors $(f^*,f_*)$ between $\mathcal{C}$ and $\mathcal{D}$
	\[\xymatrix{
		\mathcal{C} \ar@/_/[d]_{f^*}  \\
		\mathcal{D} \ar@/_/[u]_{f_*}
	}\]
	where $f^*$ is also a tensor functor. We say that the pair $(f^*,f_*)$ \emph{satisfies the projection formula} if for all $D \in \mathcal{D},C \in \mathcal{C}$ there are isomorphisms
	\[C \otimes_{\mathcal{C}} f_*(D) \cong f_*(f^*(C) \otimes_{\mathcal{D}} D)\]
	which are natural in both variables.
	\label{defprojform}
\end{dfn}

\begin{rem}
	The situation of Definition \ref{defprojform} is not restricted to algebraic geometry, see e.g.\ Theorem \ref{thmresindprojform} or \cite{balmersplit}*{Proposition 1.2} for a more general result.
\end{rem}

\begin{lma}
	Let $(f^*,f_*)$ be a pair of functors between $\mathcal{C}$ and $\mathcal{D}$ that satisfies the projection formula. Assume that $\dim\left(\supp\left(f_*(\mathbb{I}_{\mathcal{D}})\right)\right) \neq \pm \infty$. Then the functor $f_* \circ f^*$ has relative dimension $0$.
	\label{lmaprojformreldim}
\end{lma}
\begin{proof}
	Using that $(f^*,f_*)$ satisfies the projection formula, we have an isomorphism 
	\[f_* \circ f^* (C) \cong f_*(\mathbb{I}_{\mathcal{D}}) \otimes C \]
	for all objects $C$ of $\mathcal{C}$. The result then follows by Proposition \ref{proptensorreldimzero}.
\end{proof}

\begin{cor}
	Let $(f^*,f_*)$ be a pair of functors between $\mathcal{C}$ and $\mathcal{D}$ that satisfies the projection formula and assume $\dim\left(\supp\left(f_*(\mathbb{I}_{\mathcal{D}})\right)\right) \neq \pm \infty$. Furthermore assume $f^*$ and $f_*$ have relative dimensions $\dim(f^*)$ and $\dim(f_*)$ respectively. Then
	\[\dim(f^*) + \dim(f_*) \geq 0\]
	\label{reldimsum}
\end{cor}
\begin{proof}
	This is an immediate consequence of the fact that 
	\[\dim(f_* \circ f^*) \leq \dim(f^*) + \dim(f_*)\] 
	(see Remark \ref{reldimrem}) and Lemma \ref{lmaprojformreldim}.
\end{proof}

Let us give two examples from algebraic geometry, which show that functors with a relative dimension occur naturally.
\subsection{Example: flat pullback.} 
We fix $X,Y$ separated schemes of finite type over a field and $f: X \to Y$ a faithfully flat morphism. We consider $\Dperf(X)$ and $\Dperf(Y)$ with the standard structure of tensor triangulated categories and assume that they are equipped with the opposite of the Krull codimension function~$-\codim_{\mathrm{Krull}}$. We recall the following helpful lemma:
\begin{lma}[see \cite{EGA4}*{Chap. 2, Corollaire 6.1.4}]
	For all closed subsets $Z \subset Y$, we have 
	\[\codim(Z) = \codim(f^{-1}(Z))~.\]
	\label{lmacodimpreimequal}
\end{lma}

\begin{lma}
	The functor $\mathrm{L}f^{*}: \Dperf(Y) \to \Dperf(X)$ has relative dimension~$0$.
\end{lma}
\begin{proof}
	We need to check that for every $A^{\bullet} \in \Dperf(Y)_{(p)}$, the complex $\mathrm{L}f^*(A^{\bullet})$ is contained in $\Dperf(X)_{(p)}$. Thus, assume that
	\[-\codim\left(\supp\left( \bigoplus_i \mathrm{H}^i(A^{\bullet}) \right)\right) = q \leq p~.\]
	As $f$ is flat, $f^*$ is exact, and so we have
	\[\bigoplus_i \mathrm{H}^i\left(\mathrm{L}f^*\left(A^{\bullet}\right)\right) = \bigoplus_i f^*\left(\mathrm{H}^i\left(A^{\bullet}\right)\right)\]
	This implies that
	\begin{align*}
	-\codim\left(\supp\left( \bigoplus_i \mathrm{H}^i(\mathrm{L}f^*(A^{\bullet})) \right)\right) &= -\codim\left(\bigcup_i \supp\left( f^*\left(\mathrm{H}^i (A^{\bullet}) \right) \right)\right)\\ 
	&= -\codim\left(f^{-1}\left( \bigcup_i \supp\left(\mathrm{H}^i (A^{\bullet}) \right) \right)\right)\\
	&= q \leq p
	\end{align*}
	where the last equality follows from Lemma \ref{lmacodimpreimequal}. Setting $q=p$ shows that $\mathrm{L}f^{*}(\Dperf(Y)_{(p)})$ is not contained in $\Dperf(X)_{(p-1)}$ and finishes the proof of the statement.
\end{proof}

Using Theorem \ref{functoriality}, we know now that $\mathrm{L}f^*$ induces morphisms between the tensor triangular cycle and Chow groups of $\Dperf(Y)$ and~$\Dperf(X)$. We want to show that in the non-singular case, these correspond to the flat pullback homomorphisms on Chow groups defined in \cite{fulton}. We say that a flat morphism $g: X \to Y$ \emph{has relative dimension $r$}, if for all closed subvarieties $V \subset Y$ and all irreducible components $V'$ of $g^{-1}(V)$ we have that $\dim(V') = \dim(V) + r$. For such a morphism $g$, Fulton defines in \cite{fulton}*{Chapter 1.7} a pullback  homomorphism 
\[g^{*}:\Chow_n(Y) \to \Chow_{n+r}(X)\] 
for $0 \leq n \leq \dim(Y)$ by sending $[V] \in \Chow_n(Y)$ to $[g^{-1}(V)] \in \Chow_{n+r}(X)$, the class of the scheme-theoretic inverse image of $V$ under $g$. 

We now assume that $f: X \to Y$ is faithfully flat of relative dimension $r$.
\begin{thm}
	Assume that $X,Y$ are non-singular and for $S=X,Y$, let 
	\[\rho_S: \Chow^{\Delta}_p(\Dperf(S)) \to \Chow^{-p}(S)\] 
	be the isomorphisms from Proposition \ref{propagreement}. Then for all $p$, there is a commutative diagram
	\[
	\xymatrix@C+30pt{
		\Chow^{\Delta}_p(\Dperf(Y)) \ar[r]^{\Chow^{\Delta}_p(\mathrm{L}f^*)} \ar[d]^{\rho_Y} & \Chow^{\Delta}_p(\Dperf(X)) \ar[d]^{\rho_X} \\
		\Chow^{-p}(Y) \ar[r]^{f^*} & \Chow^{-p}(X)
	}
	\]
	where $f^*$ denotes the flat pullback homomorphism on the usual Chow group. (cf. \cite{fulton}*{Chapter 1.7}).
	\label{pullbackagrees}
\end{thm}
\begin{proof}
	As both $f^*$ and $\Chow^{\Delta}_p(\mathrm{L}f^*)$ are induced by the corresponding morphisms on the cycle level, it is enough to check that the diagram
	\[
	\xymatrix@C+30pt{
		\Cyc^{\Delta}_p(\Dperf(Y)) \ar[r]^{\Cyc^{\Delta}_p(\mathrm{L}f^*)} \ar[d]^{\rho_Y}& \Cyc^{\Delta}_p(\Dperf(X)) \ar[d]^{\rho_X} \\
		Z^{-p}(Y) \ar[r]^{f^*} & Z^{-p}(X)
	}
	\]
	commutes. In order to do this, let $Z \subset Y$ be a subvariety (=reduced and irreducible subscheme) of $Y$ of codimension $-p$, with associated ideal sheaf $I_Z$ and cycle $[Z] \in Z^{-p}(Y)$. Consider the class $[W^{\bullet}]$ in
	\begin{align*}
	\Cyc^{\Delta}_p(\Dperf(Y)) &= \Kzero\left((\Dperf(Y)_{(p)}/\Dperf(Y)_{(p-1)})^{\natural}\right) \\
	&= \Kzero(\Dperf(Y)_{(p)}/\Dperf(Y)_{(p-1)})
	\end{align*}
	where $W^{\bullet}$ is the complex concentrated in degree zero with~$\mathrm{H}^0(W^{\bullet}) = \mathcal{O}_Y/I_Z =: \mathcal{O}_Z$. Then we have~$\rho_Y([W^{\bullet}]) = Z \in Z^{-p}(Y)$: indeed, using Proposition \ref{propagreement} we calculate
	\[\rho_Y([W^{\bullet}]) = \sum_i \sum_{P \in Y^{(-p)}} (-1)^i \length_{\mathcal{O}_{Y,P}} \left(\mathrm{H}^i(W^{\bullet})_P\right) \cdot \overline{\lbrace P \rbrace}\]
	where $\mathrm{H}^i(W^{\bullet})_P$ is the stalk of the $i$-th cohomology sheaf of the complex $W^{\bullet}$ at the point~$P$. Using that $W^{\bullet}$ is concentrated in degree zero and that $\length_{\mathcal{O}_{Y,P_Z}}(\mathcal{O}_{Z,P_Z})$ is equal to 1, where $P_Z$ is the generic point of $Z$, we see that~$\rho_Y([W^{\bullet}]) = Z$.
	
	Furthermore, using that $f$ is flat, we compute that $\Cyc^{\Delta}_p(\mathrm{L}f^*)\left([W^{\bullet}]\right) = [U^{\bullet}]$, where $U^{\bullet}$ is the complex of sheaves concentrated in degree zero with $\mathrm{H}^0(U^{\bullet}) = \mathcal{O}_X/I_{f^{-1}(Z)}$ and $f^{-1}(Z)$ denotes the scheme-theoretic inverse image of $Z$ under~$f$. Clearly we have  $\rho_X([U^{\bullet}]) = [f^{-1}(Z)]$, the cycle associated to the scheme-theoretic inverse image of $Z$, and so we conclude that 
	\[\rho_X \circ \Cyc^{\Delta}_p(\mathrm{L}f^*) \circ \rho_Y^{-1} ([Z]) = [f^{-1}(Z)] = f^*[Z]\]
	By additivity of the four maps in the diagram the theorem follows.
\end{proof}

\subsection{Example: proper push-forward.}
\label{example:properpfwd}
Let $X$ and $Y$ be integral, non-singular, separated schemes of finite type over an algebraically closed field. (The latter assumption will be needed in order to use \cite{serrealgloc}*{Proposition V.C.6.2}). Let $f: X \to Y$ denote a proper morphism. We consider $\Dperf(X),\Dperf(Y)$ with the standard structure of tensor triangulated categories, \emph{but this time we choose $\dim_{\mathrm{Krull}}$ as a dimension function}. Note that this implies $\Chow^{\Delta}_p(\Dperf(S)) \cong \Chow^{\dim(S) - p}(S)$ for~$S = X,Y$.

As $f$ is proper, we obtain a functor $\mathrm{R}f_*: \Db(\mathrm{Coh}(X)) \to \Db(\mathrm{Coh}(Y))$ (see e.g.\ \cite{huybrechts}*{Theorem 3.23}), and so our regularity assumptions on $X$ and $Y$ imply that we also get a functor $\mathrm{R}f_*:\Dperf(X) \to \Dperf(Y)$. 

\begin{lma}
	The functor $\mathrm{R}f_*:\Dperf(X) \to \Dperf(Y)$ has relative dimension~$0$.
\end{lma}
\begin{proof}
	Let $A^{\bullet}$ be a complex in $\Dperf(X)$ such that~$\dim\left(\supp\left( \bigoplus_i \mathrm{H}^i(A^{\bullet}) \right)\right) \leq d$. There is a spectral sequence 
	\[E_2^{p,q} = \mathrm{R}^pf_*(\mathrm{H}^q(A^{\bullet})) \Longrightarrow \mathrm{H}^{p+q}(\mathrm{R}f_*(A^{\bullet}))\]
	(see for example \cite{huybrechts}*{p.74 (3.4)}) that converges, as $A^{\bullet}$ is bounded. By assumption, all the cohomology sheaves $\mathrm{H}^q(A^{\bullet})$ are supported in dimension $\leq d$, and by \cite{serrealgloc}*{Proposition V.C.6.2 (a)}, we therefore have $\dim(\supp(\mathrm{R}^pf_*(\mathrm{H}^q(A^{\bullet})))) \leq d$ for all~$p$. This implies that the terms $E^{p,q}_{\infty}$ are supported in dimension $\leq d$ as well. Therefore, all objects $\mathrm{H}^{p+q}(\mathrm{R}f_*(A^{\bullet}))$ admit a finite filtration such that the subquotients are supported in dimension~$\leq d$. An induction argument then shows that the same must hold for $\mathrm{H}^{p+q}(\mathrm{R}f_*(A^{\bullet}))$. We conclude that 
	\[\dim\left(\supp\left(\bigoplus_i \mathrm{H}^i(\mathrm{R}f_*(A^{\bullet})) \right)\right) \leq d\]
	which shows that $\mathrm{R}f_*(A^{\bullet}) \in \Dperf(Y)_{(d)}$. In order to show that the relative dimension of $\mathrm{R}f_*$ is $0$, we need to show that there is a $B^{\bullet} \in \Dperf(X)$ such that~$\dim(\supp(B^{\bullet})) = \dim(\supp(\mathrm{R}f_*(B^{\bullet})))$. If $P$ is any closed point of $X$ with associated ideal sheaf $I_{P}$, then the complex $C_P^{\bullet}$ concentrated in degree 0 with $\mathcal{O}_X/I_P$ has~$\dim(\supp(C_P^{\bullet})) = 0$. By the result we just proved, $\mathrm{R}f_*(C_P^{\bullet}) \in \Dperf(Y)_{(0)}$, which implies that either $\mathrm{R}f_*(C_P^{\bullet}) = 0$ or~$\dim(\supp(\mathrm{R}f_*(C_P^{\bullet}))) = 0$. If $\mathrm{R}f_*(C_P^{\bullet}) = 0$, we would certainly have $\mathrm{H}^0(\mathrm{R}f_*(C_P^{\bullet})) = 0$, but this is impossible by the spectral sequence we used above: indeed, it is easy to see that $E^{0,0}_{\infty} = E^{0,0}_2$ as $\mathrm{H}^i(C_P^{\bullet}) = 0$ for~$i \neq 0$. But we have~$E^{0,0}_2 = \mathrm{R}^0f_*(C_P^{\bullet}) = f_*(\mathcal{O}_X/I_P) \neq 0$. Thus $\mathrm{H}^0(\mathrm{R}f_*(C_P^{\bullet}))$ has a non-zero subquotient from which we deduce that~$\mathrm{R}f_*(C_P^{\bullet}) \neq 0$. We conclude that $\dim(\supp(\mathrm{R}f_*(C_P^{\bullet}))) = 0$ which completes the proof.
\end{proof}

The previous lemma establishes that $\mathrm{R}f_*$ induces homomorphisms 
\[\Chow^{\Delta}_{p}(\Dperf(X)) \to \Chow^{\Delta}_{p}(\Dperf(Y))\] 
for all~$p$. Again, we can show that these are exactly the ones we would expect.
\begin{thm}
	Denote by $\rho_S: \Chow^{\Delta}_p(\Dperf(S)) \to \Chow^{\dim(S) - p}(S)$ for $S=X,Y$ the isomorphisms from Proposition \ref{propagreement}. Then for all $p$, there is a commutative diagram
	\[
	\xymatrix@C+30pt{
		\Chow^{\Delta}_p(\Dperf(X)) \ar[r]^{\Chow^{\Delta}_p(\mathrm{R}f_{*})} \ar[d]^{\rho_X} & \Chow^{\Delta}_p(\Dperf(Y)) \ar[d]^{\rho_Y} \\
		\Chow^{\dim(X) - p}(X) \ar[r]^{f_*} & \Chow^{\dim(Y)-p}(Y)
	}
	\]
	where $f_*$ denotes the proper push-forward homomorphism on the usual Chow group (cf. \cite{fulton}*{Chapter 1.4}).
	\label{proppushforwardagrees}
\end{thm}
\begin{proof}
	Again, it suffices to show the statement for the maps on the cycle groups, as the maps on the Chow groups are induced by those. By additivity of the four maps in the diagram it is enough to check that for an (integral) subvariety $V \subset X$ of dimension $p$ and an element $v \in Z_p^{\Delta}(\Db(X))$ with $\rho_X(v) = [V]$, we have~$\rho_Y \circ \Cyc^{\Delta}_p(\mathrm{R}f_{*}) (v) = f_*([V])$. So, fix $V$ as above and consider the complex of coherent sheaves $W^{\bullet}$ that is concentrated in degree 0 and has $\mathrm{H}^0(W^{\bullet}) = \mathcal{O}_V$, where $\mathcal{O}_V = \mathcal{O}_X/\mathcal{I}_V$ and $\mathcal{I}_V$ is the ideal sheaf associated to~$V$. The complex $W^{\bullet}$ represents a class $[W^{\bullet}]$ in 
	\[\Cyc_p^{\Delta}(\Db(X)) = \Kzero\left(\Db(X)_{(p)}/\Db(X)_{(p-1)}\right) = \Kzero\left(\left(\Db(X)_{(p)}/\Db(X)_{(p-1)}\right)^{\natural}\right)\] 
	and similarly to the calculation in Theorem \ref{pullbackagrees} we see that~$\rho_X([W^{\bullet}]) = V$.
	
	For the next step, we compute
	\begin{align*}
	\rho_Y \circ \Cyc^{\Delta}_p(\mathrm{R}f_*) ([W^{\bullet}]) &= \sum_i \sum_{Q \in Y_{(p)}} (-1)^i \length_{\mathcal{O}_{Y,Q}} \left(\mathrm{H}^i(\mathrm{R}f_*(W^{\bullet}))_Q\right) \cdot \overline{\lbrace Q \rbrace} \\ 
	&= \sum_i \sum_{Q \in Y_{(p)}} (-1)^i \length_{\mathcal{O}_{Y,Q}} \left(\mathrm{R}^if_*(\mathcal{O}_V)_{Q}\right) \cdot \overline{\lbrace Q \rbrace} \\
	&= \sum_i (-1)^i \sum_{Q \in Y_{(p)}} \length_{\mathcal{O}_{Y,Q}} \left(\mathrm{R}^if_*(\mathcal{O}_V)_Q\right) \cdot \overline{\lbrace Q \rbrace} .
	\end{align*}
	Using \cite{serrealgloc}*{Proposition V.C.6.2 (b)}, we see that this is equal to $f_*([V])$, which means that we have shown $\rho_Y \circ \Cyc^{\Delta}_p(Rf_*) ([W^{\bullet}]) = f_*([V])$ and thus have finished the proof of the theorem.
\end{proof}

\section{An alternative definition of rational equivalence}
\label{sectaltrat}
\subsection{Divisors and tensor triangular Chow groups}
Instead of choosing the $\mathrm{K}$-theoretic approach of Definition \ref{defchow} in order to obtain a notion of rational equivalence, one can try to imitate the original construction from algebraic geometry of taking divisors of functions on subvarieties. Following \cite{balmerchow}, we can define ``divisors of functions'' in the categorical context.

\begin{conv}
	For the rest of the section, $\mathcal{K}$ denotes a tensor triangulated category that is rigid and such that $\Spc(\mathcal{K})$ is a noetherian topological space. We also fix a dimension function on $\mathcal{K}$.
\end{conv}

Let $Q^{\natural}$ denote the composition of the Verdier quotient functor 
\[\mathcal{K}_{(p)} \to \mathcal{K}_{(p)}/\mathcal{K}_{(p-1)}\] 
and the inclusion into the idempotent completion
\[\mathcal{K}_{(p)}/\mathcal{K}_{(p-1)} \to (\mathcal{K}_{(p)}/\mathcal{K}_{(p-1)})^{\natural}~.\]
The functor $Q^{\natural}$ induces a group homomorphism 
\begin{align*}
q^{\natural}: \Kzero(\mathcal{K}_{(p)}) &\to \Kzero\left((\mathcal{K}_{(p)}/\mathcal{K}_{(p-1)})^{\natural}\right) = \Cyc^{\Delta}_p(\mathcal{K}) \\
[a] &\mapsto [Q^{\natural}(a)]
\end{align*}

For an object $a$ in the Verdier quotient $\mathcal{K}_{(p+1)}/\mathcal{K}_{(p)}$ and an automorphism 
\[f:a \to a~,\]
choose a fraction $a \xleftarrow{\beta} b \xrightarrow{\alpha} a$ in $\mathcal{K}_{(p+1)}$ representing~$f$. We will then have $\cone(\beta) \in \mathcal{K}_{(p)}$ by definition of the Verdier quotient. We also must have  $\cone(\alpha) \in \mathcal{K}_{(p)}$: indeed, $\alpha$ must be an isomorphism in $\mathcal{K}_{(p+1)}/\mathcal{K}_{(p)}$ as the composition $\alpha \circ \beta^{-1} = f$ is one, and thus its cone must be zero in $\mathcal{K}_{(p+1)}/\mathcal{K}_{(p)}$. This implies that it is in $\mathcal{K}_{(p)}$, as the latter is a thick subcategory.

We then define the \emph{divisor of $f$} (cf. \cite{balmerchow}) as
\[\divis^{\Delta}(f) := q^{\natural}([\cone(\alpha)] - [\cone(\beta)]) = [Q^{\natural}(\cone(\alpha))] - [Q^{\natural}(\cone(\beta))] \]
The following shows that $\divis^{\Delta}(f)$ is well-defined.
\begin{prop}
	The expression $\divis^{\Delta}(f)$ does not depend on the choice of $\alpha$ and~$\beta$.
\end{prop}
\begin{proof}
	If we have an equivalent fraction $a \xleftarrow{\beta'} c \xrightarrow{\alpha'} a$, there is by definition a commutative diagram in $\mathcal{K}_{(p+1)}$
	\[
	\xymatrix{
		& b \ar[ld]_{\beta} \ar[rd]^{\alpha}& \\
		a & d \ar[r]^{x} \ar[l]^{y} \ar[u]^{f} \ar[d]^{g} & a  \\
		& c \ar[ul]^{\beta'} \ar[ur]_{\alpha'} &.
	}
	\]
	Using the octahedral axiom, we obtain the following distinguished triangles in  $\mathcal{K}_{(p)}$:
	\begin{align*}
	\cone(f) \to \cone(y) \to \cone(\beta) &\to \Sigma(\cone(f)) \\
	\cone(g) \to \cone(y) \to \cone(\beta') &\to \Sigma(\cone(g)) \\
	\cone(f) \to \cone(x) \to \cone(\alpha) &\to \Sigma(\cone(f)) \\
	\cone(g) \to \cone(x) \to \cone(\alpha') &\to \Sigma(\cone(g))
	\end{align*}
	These show that $[\cone(\alpha)] - [\cone(\alpha')]$ and $[\cone(\beta)] - [\cone(\beta')]$ are both equal to the element $[\cone(g)] - [\cone(f)]$ in $\Kzero(\mathcal{K}_{(p)})$. Thus, we have 
	\[[\cone(\alpha)] - [\cone(\beta)] = [\cone(\alpha')] - [\cone(\beta')]\]
	in $\Kzero(\mathcal{K}_{(p)})$. Applying the homomorphism $q^{\natural}$ on both sides of the equation yields the statement.
\end{proof}

We now define alternative Chow groups as ``cycles modulo divisors of functions''.
\begin{dfn} Let $\mathfrak{I}$ denote the subgroup of $\Cyc^{\Delta}_p(\mathcal{K})$ generated by all expressions~$\divis^{\Delta}(f)$, where $f$ runs over all automorphisms of all objects of $\mathcal{K}_{(p+1)}/\mathcal{K}_{(p)}$. Then define
	\[\chow^{\Delta}_p(\mathcal{K}) := \Cyc^{\Delta}_p(\mathcal{K})/\mathfrak{I}~.\]
	\label{dfnaltchow}
\end{dfn}

Let us now investigate the relation between $\chow^{\Delta}_p(\mathcal{K})$ and $\Chow^{\Delta}_p(\mathcal{K})$. Recall from Definition \ref{defchow} that $\Chow^{\Delta}_p(\mathcal{K}) = \Cyc^{\Delta}_p(\mathcal{K})/j \circ q (\ker(i))$ where $i,q,j$ are taken from the diagram
\[
\xymatrix{
	\Kzero(\mathcal{K}_{(p)}) \ar[r]^i \ar@{->>}[d]^q& \Kzero(\mathcal{K}_{(p+1)}) \\
	\Kzero(\mathcal{K}_{(p)}/\mathcal{K}_{(p-1)}) \ar@{^{(}->}[r]^(0.38){j} & \Kzero\left( (\mathcal{K}_{(p)}/\mathcal{K}_{(p-1)})^{\natural} \right) = \Cyc^{\Delta}_{p}(\mathcal{K}) ~.
}
\]
\begin{prop}
	We have an inclusion $\mathfrak{I} \subset j \circ q (\ker(i))$.
	\label{propaltchowinclusion}
\end{prop}
\begin{proof}
	If $f:a \to a$ is an isomorphism in $\mathcal{K}_{(p+1)}/\mathcal{K}_{(p)}$ represented by a fraction 
	\[a \xleftarrow{\beta} b \xrightarrow{\alpha} a\]
	in $\mathcal{K}_{(p+1)}$, then in $\Kzero(\mathcal{K}_{(p+1)})$, we have 
	\[[\cone(\alpha)] = [\cone(\beta)] = [b] - [a]\]
	and thus $[\cone(\alpha)] - [\cone(\beta)] = 0$. Therefore, $[\cone(\alpha)] - [\cone(\beta)]$ will certainly be in $\ker(i)$. The statement then follows as $q^{\natural} = j \circ q$.
\end{proof}

\begin{cor}
	For all $p \in \mathbb{Z}$, there is an epimorphism
	\[\chow^{\Delta}_p(\mathcal{K}) \to \Chow^{\Delta}_p(\mathcal{K})~.\]
\end{cor}
\begin{proof}
	This is an immediate consequence of Proposition \ref{propaltchowinclusion}.
\end{proof}

It is not clear to the author if the inclusion $\mathfrak{I} \supset j \circ q (\ker(i))$ holds in general, so $\chow^{\Delta}_p(\mathcal{K})$ and $\Chow^{\Delta}_p(\mathcal{K})$ are a priori different. 

\subsection{Varieties with an ample line bundle}
We will now prove that both groups $\chow^{\Delta}_p(\Dperf(X))$ and $\Chow^{\Delta}_p(\Dperf(X))$ coincide with $\Chow^{-p}(X)$ when $X$ is separated, non-singular of finite type over a field and has an ample line bundle.
\begin{thm}
	Let $X,\Dperf(X)$ be as in Convention \ref{convschemeass} and assume furthermore that $X$ has an ample line bundle~$\mathcal{L}$. Then there are isomorphisms
	\[\chow^{\Delta}_p(\Dperf(X)) \cong \Chow^{\Delta}_p(\Dperf(X)) \cong \Chow^{-p}(X)\]
	for all $p \in \mathbb{Z}$.
	\label{altchowagree}
\end{thm}
\begin{proof}
	Using Theorem \ref{agreementthm} and Proposition \ref{propaltchowinclusion}, we already know that the subgroup $\mathfrak{I}$ is contained in the subgroup of cycles rationally equivalent to zero. Thus, it suffices to show that any cycle rationally equivalent to zero can be obtained as $\divis^{\Delta}(f)$ for some object~$a \in \Db(X)_{(p+1)}/\Db(X)_{(p)}$ and morphism $f\in \Aut(a)$. The essential point is that for a subvariety $V \subset X$ of codimension $-(p+1)$ we can write the function field of $V$ as
	\[k(V) = \left(\bigoplus_{i \geq 0} \Gamma\left(X, \mathcal{O}_V \otimes \mathcal{L}^{\otimes i}\right)\right)_{((0))}~,\]
	where $\mathcal{O}_V := \mathcal{O}_X/\mathcal{I}_V$ and $\mathcal{I}_V$ is the ideal sheaf associated to~$V$. Indeed, this is a consequence of \cite{ega2}*{Th\'eor\`eme 4.5.2} and the fact that the restriction of an ample line bundle to a closed subscheme is ample.
	
	Thus, for $h \in k(V)$, we can write $h = f/g$ with $f,g \in \Gamma(X, \mathcal{O}_V \otimes \mathcal{L}^{\otimes n})$ for some~$n \in \mathbb{N}$. From this, we obtain exact sequences
	\begin{equation}
		0 \to \mathcal{O}_V \xrightarrow{m_f} \mathcal{O}_V \otimes \mathcal{L}^{\otimes n} \to \coker(m_f) \to 0
		\label{eqsuppval1}
	\end{equation}
	and
	\begin{equation}
		0 \to \mathcal{O}_V \xrightarrow{m_g} \mathcal{O}_V \otimes \mathcal{L}^{\otimes n} \to \coker(m_g) \to 0
		\label{eqsuppval2}
	\end{equation}
	where $m_f, m_g$ are the obvious multiplication maps. By using the local isomorphisms $\mathcal{L}^{\otimes n}|_{U_i} \cong \mathcal{O}_{X}|_{U_i}$ for some open cover $\lbrace U_i \rbrace_{i \in I}$, we obtain that 
	\[\supp(\coker(m_f)) = V(f) \subset V\]
	and
	\[\supp(\coker(m_g)) = V(g) \subset V~.\]
	We have
	\[\codim(V(f))=\codim(V(g)) = -p\]
	which one deduces from looking at the local rings of the generic points of the irreducible components of $V(f),V(g)$ and applying Krull's principal ideal theorem. If we interpret the exact sequences (\ref{eqsuppval1}) and (\ref{eqsuppval2}) as distinguished triangles in $\Dperf(X)_{(p+1)}$, this shows that $\cone(m_f),\cone(m_g) \in \Dperf(X)_{(p)}$. We can use them to form a fraction and obtain an automorphism
	\[\hat{h}: \mathcal{O}_V \otimes \mathcal{L}^{\otimes n} \xleftarrow{m_g} \mathcal{O}_V \xrightarrow{m_f} \mathcal{O}_V \otimes \mathcal{L}^{\otimes n} \]
	in the quotient category $\Dperf(X)_{(p+1)}/\Dperf(X)_{(p)}$. Using the explicit formula from Proposition~\ref{propagreement} one obtains
	\begin{align*}
		\rho_X\left(\divis^{\Delta}(\hat{h})\right) &= \sum_i \sum_{x \in X^{(-p)}} (-1)^i \left(\length_{\mathcal{O}_{X,x}}\left(\mathrm{H}^i\left(\cone(m_f)\right)_x\right) - \length_{\mathcal{O}_{X,x}}\left(\mathrm{H}^i\left(\cone(m_g)\right)_x\right)\right) \cdot \overline{\lbrace x \rbrace}\\
		&= \sum_{x \in X^{(-p)}} \left(\length_{\mathcal{O}_{X,x}}\left(\coker(m_f)_x\right) - \length_{\mathcal{O}_{X,x}}\left(\coker(m_g)_x\right)\right) \cdot \overline{\lbrace x \rbrace}~,
	\end{align*}
	which is equal to $\divis(h)$ by definition and therefore proves the statement.
\end{proof}

\subsection{Functoriality}
The groups $\chow^{\Delta}_p$ possess the same functoriality properties as $\Chow^{\Delta}_p$. 
\begin{prop}
	A functor $F: \mathcal{K} \to \mathcal{L}$ of relative dimension $n$ induces homomorphisms 
	\[\chow^{\Delta}_{p}(F): \chow^{\Delta}_{p}(\mathcal{K}) \to \chow^{\Delta}_{p+n}(\mathcal{L})\] 
	for all~$p$.
\end{prop}
\begin{proof}
	We have already seen in Theorem \ref{functoriality} how $F$ induces a morphism of cycle groups, so it remains to check that $F$ sends divisors to divisors. In order to do this, let $D = \divis(h)$ be the divisor of an automorphism $h: a \to a$ in $\mathcal{K}_{(p+1)}/\mathcal{K}_{(p)}$. Assume that a corresponding fraction for $h$ is given by 
	\[a \xleftarrow{g} b \xrightarrow{f} a,\] 
	where $\cone(f),\cone(g) \in \mathcal{K}_{(p)}$.
	
	As $F$ has relative dimension $n$, we have $F(a) \in \mathcal{L}_{(p+n+1)}$ and 
	\[\cone(F(f)), \cone(F(g)) \in \mathcal{L}_{(p+n)}~,\]
	so there is an automorphism $\hat{h}$ of $F(a)$ in $\mathcal{L}_{(p+n+1)}/\mathcal{L}_{(p+n)}$ given by the fraction 
	\[F(a) \xleftarrow{F(g)} F(b) \xrightarrow{F(f)} F(a)\]
	and it makes sense to define $E := \divis (\hat{h})$. To see that $\Cyc^{\Delta}_p(F)$ sends $D$ to $E$, simply note that $D$ is given as the element 
	\[[\cone(f)] - [\cone(g)] \in \Kzero(\mathcal{K}_{(p)}/\mathcal{K}_{(p-1)}) \subset \Kzero\left((\mathcal{K}_{(p)}/\mathcal{K}_{(p-1)})^{\natural} \right) = \Cyc^{\Delta}_p(\mathcal{K})\]
	and that
	\begin{align*}
	\Cyc^{\Delta}_p(F)([\cone(f)] - [\cone(g)]) &= [F(\cone(f))] - [F(\cone(g))] \\
	&= [\text{cone}(F(f))] - [\cone(F(g))] = E~.
	\end{align*}
\end{proof}

\section[Tensor triangular Chow groups in modular representation theory]{Tensor triangular cycle groups and Chow groups in modular representation theory}
Up to this point we have almost exclusively considered examples from algebraic geometry. However, tensor triangulated categories also occur in different contexts. One of these is modular representation theory, where one studies $kG$-modules for a finite group $G$ and a field $k$ such that $\text{char}(k)$ divides~$|G|$. A useful tool in this context is the stable category $kG\stab$, which is obtained as the stable category of $kG\Mod$, the Frobenius category of finitely generated left $kG$-modules. The category $kG\stab$ is a tensor triangulated category. By a theorem of Rickard (see \cite{rickardstable}), it is closely related to $\Db(kG\Mod)$, the bounded derived category of finitely generated $kG$-modules, which is also tensor triangulated. Using the theory from the previous sections, we therefore have a notion of tensor triangular Chow groups for these categories. In this section we compare the tensor triangular Chow groups of $kG\stab$ and $\Db(kG\Mod)$, compute concrete examples of these groups and show that stable induction and restriction functors fit in the framework of functors with a relative dimension.

\subsection{Basic definitions and results}
We recall some basic definitions and results that we will need. All of them can be found in the books by Carlson \cite{carlson} and Benson \cites{bensonrep1,bensonrep2} or in Balmer's article \cite{balmer2005spectrum}. For the rest of the chapter, $G$ will denote a finite group, $k$ is a field of characteristic $p$ dividing $|G|$, and $kG$ is the corresponding group algebra. Associated to this algebra is the abelian category $kG\Mod$ consisting of the finitely-generated left $kG$-modules. Given two modules $M,N \in kG\Mod$, we can form their tensor product $M \otimes_k N$, which is again a finitely-generated left $kG$-module when we consider it with the diagonal action 
\[g(m \otimes n) := gm \otimes gn\] 
for $g \in G, m \in M$ and $n \in N$ and extend linearly. Furthermore, $\Hom_k(M,N)$, the set of $k$-linear maps from $M$ to $N$ can be made a finitely generated $kG$-module by setting
\[(gf)(m) := f(g^{-1}m)\]
for $g \in G, m \in M$ and $n \in N$ and extending linearly.

The category $kG\Mod$ is a Frobenius category, and so we can form the associated stable category $kG\stab$ which is naturally triangulated. It can be given a symmetric monoidal structure with the tensor product induced by $- \otimes_k -$ with unit object $k$, the trivial $kG$-module. Thus, $kG\stab$ is an essentially small, tensor triangulated category. It is also rigid, where the dual of an object $M$ is given as~$\Hom_k(M,k)$. 
\begin{dfn}
	The \emph{cohomology ring} of $kG$ is defined as the graded ring
	\[\mathrm{H}^*(G,k) := \bigoplus_{i \geq 0} \Ext^i_{kG}(k,k)~.\]
	The \emph{projective support variety} of $kG$ is defined as 
	\[\mathcal{V}_G(k) := \Proj(\mathrm{H}^*(G,k)) ~.\]
\end{dfn}
\begin{rem}
	When $p$ is odd, $\mathrm{H}^*(G,k)$ is in general only a \emph{graded} commutative ring, so when we write $\Proj(\mathrm{H}^*(G,k))$ we really mean $\Proj(\mathrm{H}^{\text{ev}}(G,k))$ in this case, where $\mathrm{H}^{\text{ev}}(G,k)$ is the subring of all elements of even degree. Another way to deal with this difficulty is to extend the definition of $\Proj$ to graded-commutative $k$-algebras (cf.\ \cite{balbencarl}*{Section 1}).
\end{rem}

Suppose we are given any two finite-dimensional $kG$-modules $M,N$. Then the Evens-Venkov theorem (see \cite{carlson}*{Theorem 9.1}) shows that $\bigoplus_{i \geq 0} \Ext^i_{kG}(M,N)$ is a fi\-nite\-ly generated graded module over $\mathrm{H}^*(G,k)$.
\begin{dfn}
	For a $kG$-module $M \neq 0$, define $\mathrm{J}(M) \subset \mathrm{H}^*(G,k)$ as the annihilator ideal of $\Ext^*_{kG}(M,M)$ in~$\mathrm{H}^*(G,k)$. The \emph{variety of $M$} is the subvariety of $\mathcal{V}_G(k)$ associated to~$\mathrm{J}(M)$.
	\label{dfncohsuppmod}
\end{dfn}

\begin{dfn}
	Let $M$ be in $kG\Mod$. A \emph{minimal projective resolution of $M$} is a projective resolution $P_{\bullet} \to M$ such that for every other projective resolution $Q_{\bullet} \to M$ there exists an injective chain map $(P_{\bullet} \to M) \to (Q_{\bullet} \to M)$ and a surjective chain map $(Q_{\bullet} \to M) \to (P_{\bullet} \to M)$ that both lift the identity on $M$.
\end{dfn}

\begin{thm}[see \cite{carlson}*{Theorem 4.3}]
	Let $M$ be a module in $kG\Mod$. Then $M$ has a minimal projective resolution.
\end{thm}

\begin{dfn}
	Let $M$ be in $kG\Mod$ and let $P_{\bullet} \to M$ be a minimal projective resolution. The \emph{complexity} $\mathrm{c}_G(M)$ of $M$ is defined as the least integer $s$ such that there is a constant $\kappa >0$ with
	\[\dim_k(P_n) \leq \kappa \cdot n^{s-1} \quad \text{for $n>0$} \]
\end{dfn}

The complexity of a module can be read off from its variety:
\begin{thm}[cf. \cite{bensonrep2}*{Prop.\ 5.7.2}]
	If $M$ is a finitely generated $kG$-module, then 
	\[\dim(\mathcal{V}_G(M)) = \mathrm{c}_G(M)-1.\]
\end{thm}

The projective support variety of $kG$ appears as the spectrum of $kG\stab$:
\begin{thm}[cf. \cite{balmer2005spectrum}*{Corollary 5.10}]
	There is a homeomorphism 
	\[\phi: \mathcal{V}_G(k) \longrightarrow  \Spc(kG\stab)~.\]
	Furthermore, the support of a module $M \in kG\stab$ corresponds to $\mathcal{V}_G(M)$ under this map.
	\label{kGspecissuppvar}
\end{thm}
For the rest of the section, we will take $\dim_{\mathrm{Krull}}$ (cf.\ Example \ref{dimfuncex}) as a dimension function for~$kG\stab$. By Theorem \ref{kGspecissuppvar} this coincides with the usual Krull dimension on $\mathcal{V}_G(k)$ under the homeomorphism $\rho$.

\subsection{Derived category vs.\ stable category}
\label{subsectdervsstab}
We consider $\Db(kG\Mod)$, the bounded derived category of finitely generated $kG$-modules with its natural triangulation. It becomes a tensor triangulated category with the usual extension to chain complexes of the tensor product $\otimes_k$ of $kG$-modules \emph{over $k$}.

Let us immediately state that $\Db(kG\Mod)$ and $kG\stab$ are closely related: the category $kG\stab$ arises as a Verdier quotient of $\Db(kG\Mod)$. Let $\mathrm{K^b}(kG\proj)$ denote the bounded homotopy category of finitely generated projective $kG$-modules. Since quasi-isomorphisms between bounded complexes of projective modules are the same as homotopy equivalences, $\mathrm{K^b}(kG\proj)$ embeds into $\Db(kG\Mod)$ as a full triangulated subcategory.
\begin{thm}[see \cite{rickardstable}]
	The natural functor
	\[kG\stab \to \Db(kG\Mod)/\mathrm{K^b}(kG\proj)\]
	induced by the inclusion $kG\Mod \to \Db(kG\Mod)$ is an exact equivalence of tensor triangulated categories.
\end{thm}

The following theorem tells us that the spectra of $\Db(kG\Mod)$ and $kG\stab$ differ in one point only.
\begin{thm}[see \cite{balmer2010spectra}*{Theorem 8.5}]
	We have a homeomorphism
	\[\rho: \Spc(\Db(kG\Mod)) \longrightarrow \mathrm{Spec}^{\mathrm{h}}(\mathrm{H}^*(G,k))\] 
	where $\mathrm{Spec}^{\mathrm{h}}(\mathrm{H}^*(G,k))$ is the spectrum of homogeneous prime ideals in $\mathrm{H}^*(G,k)$. Furthermore the diagram
	\[
	\xymatrix{
		\Spc(kG\stab) \ar[r]^(0.45){\Spc(q)} & \Spc(\Db(kG\Mod)) \ar[d]^{\rho} \\
		\mathrm{Proj}(\mathrm{H}^*(G,k)) \ar[u]^{\varphi} \ar@{^{(}->}[r] & \mathrm{Spec}^{\mathrm{h}}(\mathrm{H}^*(G,k))
	}
	\]
	commutes, where $\varphi$ is the homeomorphism from Theorem \ref{kGspecissuppvar}, $\Spc(q)$ is the map associated to the quotient functor 
	\[q: \Db(kG\Mod) \to \Db(kG\Mod)/\mathrm{K^b}(kG\proj) \cong kG\stab~,\] 
	and the lower arrow is the inclusion of the open subset with complement the unique closed point of $\mathrm{Spec}^{\mathrm{h}}(\mathrm{H}^*(G,k))$ corresponding to the irrelevant ideal.
	\label{thmspcstabdercomp}
\end{thm}

\begin{rem}
	It is crucial here that we consider $\Db(kG\Mod)$ with the tensor product $\otimes_k$, as opposed to $\otimes_{kG}$: there is no natural left-module structure on $M \otimes_{kG} N$ for two left $kG$-modules $M,N$. If $G$ is commutative, $\otimes_{kG}$ makes $\mathrm{K^b}(kG\proj) \subset \Db(kG\Mod)$ a tensor triangulated category, but its spectrum is much less interesting, as it is homeomorphic to the usual prime ideal spectrum $\mathrm{Spec}(kG)$.
\end{rem}

We start to compare $\Chow^{\Delta}_p(\Db(kG\Mod))$ and $\Chow^{\Delta}_p(kG\stab)$.
\begin{prop}
	Consider $kG\stab$ and $\Db(kG\Mod)$ with the Krull dimension of support as a dimension function. Then for all $p \geq 0$, the Verdier quotient functor
	\[q: \Db(kG\Mod) \to \Db(kG\Mod)/\mathrm{K^b}(kG\proj) \cong kG\stab\]
	induces isomorphisms
	\[\Cyc^{\Delta}_{p+1}(\Db(kG\Mod)) \cong \Cyc^{\Delta}_p(kG\stab)~.\]
	\label{propcycstabdersame}
\end{prop}
\begin{proof}
	First, we claim that the functor $q$ sends an object with dimension of support $p+1$ to an object with dimension of support $p$ for $p \geq 0$: observe that 
	\[\supp(q(a)) = \Spc(q)^{-1}(\supp(a))~,\]
	as for any tensor triangulated functor (see \cite{balmer2005spectrum}*{Proposition 3.6}). Then, identifying $\Spc(q)$ with the inclusion of the subspace 
	\[\mathrm{Proj}(\mathrm{H}^*(G,k)) \hookrightarrow \mathrm{Spec}^{\mathrm{h}}(\mathrm{H}^*(G,k))\]
	as in Theorem \ref{thmspcstabdercomp}, we obtain
	\[\supp(q(a))) = \Spc(q)^{-1}(\supp(a)) = \supp(a) \cap \Spc(kG\stab) \subset \Spc(\Db(kG\Mod))~.\] 
	The claim now follows since the space $\Spc(\Db(kG\Mod))$ has exactly one closed point $\lbrace 0 \rbrace \subset \Db(kG\Mod)$ more than $\Spc(kG\stab)$, which is contained in the closure of every point of $\Spc(\Db(kG\Mod))$.
	
	If $\mathcal{K} = \Db(kG\Mod)$ and $\mathcal{J} = \mathrm{K^b}(kG\proj)$, we see that 
	\begin{align*}
	\mathcal{K}_{(p+1)}/\mathcal{K}_{(p)} &\cong (\mathcal{K}_{(p+1)}/\mathcal{J})/(\mathcal{K}_{(p)}/\mathcal{J})\\
	&\cong kG\stab_{(p)}/ kG\stab_{(p-1)}~,
	\end{align*}
	and the equivalence induces one on the idempotent completions. By applying $\Kzero(-)$, we get the desired result.
\end{proof}

\begin{rem}
	We can also compute $\Cyc^{\Delta}_{0}(\Db(kG\Mod))$: 
	\[\Cyc^{\Delta}_{0}(\Db(kG\Mod)) = \Kzero\left(\Db(kG\Mod)_{(0)}\right) = \Kzero\left(\mathrm{K^b}(kG\proj)\right) \cong \Kzero(kG)~.\]
	Here, the last equality is just the usual isomorphism between the Grothendieck group of the bounded derived category of an exact category and the Grothendieck group of the exact category itself (see e.g.\ \cite{weibelkbook}*{Chapter II, Theorem 9.2.2}). The equality $\Db(kG\Mod)_{(0)} = \mathrm{K^b}(kG\proj)$ can be deduced as follows: let $q: \Db(kG\Mod) \to kG\stab$ be the Verdier quotient functor and let $a \in \Db(kG\Mod)_{(0)}$. Then $\supp(a)$ is contained in $\lbrace 0 \rbrace \subset \Spc(\Db(kG\Mod))$ and $\supp(q(a)) = \supp(a) \cap \Spc(kG\stab) = \emptyset$ which implies that $q(a) = 0 \Leftrightarrow a \in \mathrm{K^b}(kG\proj)$. On the other hand, $\mathrm{K^b}(kG\proj)$ is contained in every non-zero $\otimes$-ideal of $\Db(kG\Mod)$ (see \cite{balmer2010spectra}*{Proof of Prop.\ 8.5}), from which it follows that for any $a \in \mathrm{K^b}(kG\proj)$, we must have $\supp(a) \subset \lbrace 0 \rbrace \subset \Spc(\Db(kG\Mod)$.
\end{rem}

In order to obtain the statement of Proposition \ref{propcycstabdersame} for Chow groups instead of cycle groups, we recall the following elementary lemma about abelian groups.
\begin{lma}
	Let $f: A \to B$ be a morphism of abelian groups, $S \subset A$ be a subgroup and $\hat{f}: A/S \to B/f(S)$ be the induced morphism. Then $\ker\hat{f} = p(\ker f)$, where $p: A \to A/S$ is the canonical projection.
	\label{lmakerpreserve}
	\qed
\end{lma}

\begin{thm}
	Consider $kG\stab$ and $\Db(kG\Mod)$ with the Krull dimension of support as a dimension function. Then for all $p \geq 0$, there are isomorphisms
	\[\Chow^{\Delta}_p(kG\stab) \cong \Chow^{\Delta}_{p+1}(\Db(kG\Mod))~.\]
	\label{thmderivedstabiso}
\end{thm}
\begin{proof}
	Let $\mathcal{K} :=\Db(kG\Mod), \mathcal{T}:=kG\stab$ and consider the commutative diagram
	\[
	\xymatrix{
		\mathcal{K}_{(p+1)} \ar[r]^(0.35){\pi} \ar[d]^{\iota} \ar[rd]^{\kappa}& \left(\mathcal{K}_{(p+1)}/\mathcal{K}_{(p)}\right)^{\natural} \ar[rd]^{\psi} & \\
		\mathcal{K}_{(p+2)} \ar[rd]^{\lambda} & \mathcal{T}_{(p)} \ar[r]^(0.35){\underline{\pi}} \ar[d]^{\underline{\iota}} & \left(\mathcal{T}_{(p)}/\mathcal{T}_{(p-1)}\right)^{\natural} \\
		& \mathcal{T}_{(p+1)} 
	}
	\]
	where the diagonal functors $\kappa, \lambda$ are restrictions of the Verdier quotient 
	\[q: \Db(kG\Mod) \to kG\stab\] 
	and $\psi$ is the equivalence from the proof of Proposition \ref{propcycstabdersame}. We have 
	\[\ker(\Kzero(\underline{\iota})) = \Kzero(\kappa)(\ker(\Kzero(\iota)))\] 
	as we are in the situation of Lemma \ref{lmakerpreserve}: indeed, if $j: \mathrm{K^b}(kG\proj) \to \mathcal{K}_{(p+1)}$ denotes the inclusion functor, then we can set $A := \Kzero(\mathcal{K}_{(p+1)}), B:= \Kzero(\mathcal{K}_{(p+2)}), S := \im(\Kzero(j)), p := \Kzero(\kappa), f:= \Kzero(\iota)$ and $\hat{f} := \Kzero(\underline{\iota})$. This shows that \[\Kzero(\underline{\pi})(\ker(\Kzero(\underline{\iota}))) = \Kzero(\underline{\pi}) \circ \Kzero(\kappa)(\ker(\Kzero(\iota))) = \Kzero(\psi) \circ \Kzero(\pi)(\ker(\Kzero(\iota)))\]
	which gives the desired result.
\end{proof}

We now proceed to compute some examples of tensor triangular Chow groups coming from~$kG\stab$.

\subsection{The case $G = \mathbb{Z}/p^n\mathbb{Z}$}
We begin with the case where $G = \mathbb{Z}/p^n\mathbb{Z}$ for some prime $p$ and~$n \in \mathbb{N}$. In the following, $k$ will be any field of characteristic $p$. It follows from \cite{carlson}*{Theorem 7.3} that $\mathcal{V}_G(k)$ is a point, and so a finitely generated $kG$-module has complexity 1 if and only if it is non-projective. 

Computing the tensor triangular cycle groups for $kG\stab$ amounts to calculating 
\[\Kzero\left((kG\stab_{(i)}/kG\stab_{(i-1)})^\natural\right)~.\]
One immediately sees that there is only one non-trivial case, namely when $i=0$. Then 
\[\Cyc^{\Delta}_0(kG\stab) = \Kzero(kG\stab)~,\]
as $kG\stab$ is idempotent complete. In order to compute this Grothendieck group, we use the following result:
\begin{prop}[see \cite{tachikawak0}*{Proposition 1}]
	Let $B$ be a Frobenius $k$-algebra, let $B\Mod$ be the category of finitely generated left $B$-modules and $B\stab$ the corresponding stable category. Then it holds that
	\[\Kzero(B\stab) \cong \Kzero(B\Mod)/\langle \mathrm{proj} \rangle~,\]
	where $\langle \mathrm{proj} \rangle$ is the subgroup generated by the isomorphism classes of projective modules.
	\label{tachiwakak0}
\end{prop}

Note that $\Kzero(kG\Mod)) \cong \mathbb{Z}$, as $kG$ is a commutative local artinian ring: indeed, for modules over artinian rings, being finitely generated and having finite length are equivalent, and then the result follows by  Lemma \ref{lmalengthk0}. For local rings, projective and free modules coincide, and thus it follows from Proposition \ref{tachiwakak0} that
\[\Cyc^{\Delta}_0(kG\stab) \cong \mathbb{Z}/p^n\mathbb{Z}~.\] 
We also see that this group coincides with $\Chow^{\Delta}_0(kG\stab)$, as $\dim(\mathcal{V}_G(k)) = 0$. Summarizing, we have the following:
\begin{prop}
	Let $G = \mathbb{Z}/p^n\mathbb{Z}$ for some prime $p$ and $n \in \mathbb{N}$ and $k$ any field of characteristic~$p$. Then 
	\[\Cyc^{\Delta}_i(kG\stab) = \Chow^{\Delta}_i(kG\stab) = 0 \quad \text{for all } i \neq 0\]
	and
	\[\Cyc^{\Delta}_0(kG\stab) = \Chow^{\Delta}_0(kG\stab) = \mathbb{Z}/p^n\mathbb{Z}.\]
	\label{cyclicexample}
\end{prop}

\subsection{The case $G = \mathbb{Z}/2\mathbb{Z} \times \mathbb{Z}/2\mathbb{Z}$} \index{Klein four-group} 
If $G = \mathbb{Z}/2\mathbb{Z} \times \mathbb{Z}/2\mathbb{Z} = \langle x,y \vert x^2=y^2=1, xy = yx \rangle$ and $k$ is a field of characteristic 2, the computations become more involved.

As a consequence of \cite{carlson}*{Theorem 7.6}, we have that $\mathcal{V}_G(k) = \mathbb{P}^1$. Therefore there is a proper subcategory of $kG\stab$ consisting of the modules of complexity~$\leq 1$. In order to work with those, we need the following classification:
\begin{lma}
	All finite-dimensional indecomposable $kG$-modules of odd dimension have complexity 2.
	\label{oddcomplexity}
\end{lma}
\begin{proof}
	Let $M$ be a odd-dimensional indecomposable module. If we assume that $M$ has complexity 1, then by \cite{bensonrep2}*{Theorem 5.10.4 and Corollary 5.10.7}, $M$ must be periodic, with period~1. In other words, if $\epsilon: P \twoheadrightarrow M$ is a projective cover of $M$, then we must have~$M \cong \ker(\epsilon)$. However, since $G$ is a $2$-group, the only indecomposable projective module is the free module of rank 1 (see \cite{bensonrep1}*{Section 3.14}), which has $k$-dimension 4. Thus, if $M$ has dimension $2n+1$ and $P$ has dimension $4m$, then using that $\epsilon$ is surjective and the dimension formula, we get~$\dim_k(\ker(\epsilon)) = 4m-2n-1$. We see immediately that $\ker(\epsilon)$ cannot have dimension $2n+1$, and thus $M$ cannot have complexity~1. As it is non-projective it must therefore have complexity~2.
\end{proof}

We also see that a complementary result holds for the even-dimensional representations:
\begin{lma}
	All finite-dimensional, non-projective indecomposable $kG$-mod\-ules of even dimension have complexity~1.
	\label{evencomplexity}
\end{lma}
\begin{proof}
	It follows from \cite{cheboluklein4}*{Proposition 3.1} that a non-projective indecomposable $kG$-module of even dimension is periodic with period~1. As an immediate consequence, those modules have complexity~1.
\end{proof}

The following is a direct consequence of Lemma \ref{oddcomplexity} and Lemma~\ref{evencomplexity}:
\begin{cor}
	The indecomposable $kG$-modules of odd dimension are exactly the indecomposable modules of complexity~2. The non-projective indecomposable $kG$-modules of even dimension are exactly the indecomposable modules of complexity~1.
	\qed
	\label{complexityclasses}
\end{cor}

Using this classification, we can calculate the zero-dimensional Chow group.
\begin{lma}
	The map 
	\[[M] \mapsto \dim_k(M) \mod 4\] 
	defines an isomorphism $\Kzero(kG\stab) \to \mathbb{Z}/4\mathbb{Z}$. Furthermore, if 
	\[\alpha: \Kzero(kG\stab_{(0)}) \to \Kzero(kG\stab) \cong \mathbb{Z}/4\mathbb{Z}\]
	denotes the map induced by the inclusion functor $kG\stab_{(0)} \to kG\stab$, then 
	\[\im(\alpha) \cong \mathbb{Z}/2\mathbb{Z} \subset \mathbb{Z}/4\mathbb{Z}~.\] \label{lmak0z4z2}
\end{lma}
\begin{proof}
	The ring $kG$ is local and artinian, and thus it follows from Lemma \ref{lmalengthk0} that the map 
	\[[M] \mapsto \length(M) = \dim_k(M)\]
	defines an isomorphism $\Kzero(kG\Mod) \to \mathbb{Z}$. Therefore, the map 
	\[[M] \mapsto \dim_k(M) \mod 4\] 
	defines an isomorphism $\Kzero(kG\stab) \to \mathbb{Z}/4\mathbb{Z}$ by Lemma \ref{tachiwakak0}, as every finitely generated projective module over a local ring is free.
	
	By Corollary \ref{complexityclasses}, the image of $\alpha$ in $\Kzero(kG\stab)$ consists of exactly those classes $[M]$ where $M$ has even dimension, i.e.\ 
	\[\dim_k(M) =0 \mod 4\] 
	or 
	\[\dim_k(M) = 2 \mod 4~.\] 
	Thus, $\im(\alpha) \cong \mathbb{Z}/2\mathbb{Z}$.
\end{proof}

\begin{prop} There is an isomorphism
	\[\Chow^{\Delta}_0(kG\stab) \cong \mathbb{Z}/2\mathbb{Z}.\]
	\label{Kleinfourexample1}
\end{prop}
\begin{proof}
	By definition, 
	\[\Cyc^{\Delta}_0(kG\stab) = \Kzero\left((kG\stab_{(0)})^{\natural}\right) \cong \Kzero(kG\stab_{(0)})~,\] 
	as $kG\stab_{(-1)} = 0$ and thick subcategories of idempotent complete categories are idempotent complete themselves. Using this, we have that
	\[\Chow^{\Delta}_0(kG\stab) \cong \Cyc^{\Delta}_0(kG\stab)/\ker(\alpha)~,\] 
	where 
	\[\alpha: \Kzero(kG\stab_{(0)}) \to \Kzero(kG\stab_{(1)}) = \Kzero(kG\stab)\] 
	is the map from Lemma \ref{lmak0z4z2}. Using the isomorphism theorem for abelian groups, we conclude that
	\[\Chow^{\Delta}_0(kG\stab) \cong \im(\alpha) \cong \mathbb{Z}/2\mathbb{Z}\]
	by Lemma \ref{lmak0z4z2}.
\end{proof}

For the one-dimensional Chow group we need to work a bit harder. We first take a closer look at the quotient $\mathcal{L} := kG\stab/kG\stab_{(0)}$.
\begin{lma}
	Assume $k$ is algebraically closed. The category $\mathcal{L}$ is idempotent complete.
	\label{Lisidempcomp}
\end{lma}
\begin{proof}
	Under the additional hypothesis, it is shown in \cite{carldonwhee}*{Example 5.1} that up to isomorphism, the only indecomposable object in $\mathcal{L}$ is $k$, which has endomorphism ring $K := k(\zeta)$, a transcendental field extension of~$k$. It follows that $\mathcal{L}$ is equivalent to the category of finite-dimensional vector spaces over $K$, which is idempotent complete.
\end{proof}

This enables us to prove the following:
\begin{prop} Assume $k$ is algebraically closed. There is an isomorphism
	\[\Chow^{\Delta}_1(kG\stab) \cong \mathbb{Z}/2\mathbb{Z}.\]
	\label{Kleinfourexample2}
\end{prop}
\begin{proof}
	The sequence of triangulated categories 
	\[kG\stab_{(0)} \hookrightarrow kG\stab \to kG\stab/kG\stab_{(0)}\]
	induces an exact sequence
	\[\Kzero(kG\stab_{(0)}) \overset{\alpha}{\longrightarrow} \Kzero(kG\stab) \to \Kzero(kG\stab/kG\stab_{(0)}) \to 0\]
	where $\alpha$ is the map from Lemma \ref{lmak0z4z2}. Therefore, 
	\[\Chow^{\Delta}_1(kG\stab) \cong \Kzero\left((kG\stab/kG\stab_{(0)})^{\natural}\right) \cong \Kzero(kG\stab)/\im(\alpha) \cong \mathbb{Z}/2\mathbb{Z}\] 
	as follows from Lemma \ref{Lisidempcomp} and Lemma \ref{lmak0z4z2}.
\end{proof}

\subsection{Relative dimension of restriction and induction}
We finish the section by showing that stable induction and restriction functors from modular representation theory have a relative dimension as defined in Definition~\ref{funcreldimdef}.

\subsubsection*{Some auxiliary results from representation theory.}
Let us first recall some well-known representation-theoretic results. Let $G$ be a finite group and $k$ a field such that $\mathrm{char}(k) = p$ divides~$|G|$. If $H < G$ is a subgroup, then there is an isomorphism of left $kH$-modules
\[kG \cong \bigoplus_{H \backslash G} kH\]
which implies that the functor $\Ind^G_H: kH\Mod \to kG\Mod$ is exact, since it is given by taking the tensor product with a free module. The functor $\Res^G_H: kG\Mod \to kH\Mod$ is exact as well and $\Ind^G_H$ and $\Res^G_H$ are mutually adjoint on both sides (see e.g.~\cite{carlson}*{Proposition 3.2}). A functor between abelian categories with an exact right-adjoint preserves projective objects and therefore, both $\Ind^G_H$ and $\Res^G_H$ preserve projective modules and thus induce exact functors
\[\underline{\Ind}^G_H: kH\stab \to kG\stab\]
and
\[\underline{\Res}^G_H: kG\stab \to kH\stab~.\]
The functors $\underline{\Ind}^G_H$ and $\underline{\Res}^G_H$ are adjoint as well:
\begin{lma}
	Let 
	\[G:\mathcal{S} \to \mathcal{T}~, \quad F: \mathcal{T} \to \mathcal{S}\]
	be a pair of adjoint exact functors between triangulated categories $\mathcal{S},\mathcal{T}$ and let $\mathcal{S}' \subset \mathcal{S}, \mathcal{T}' \subset \mathcal{T}$ be thick triangulated subcategories such that $G(\mathcal{S}') \subset \mathcal{T}'$ and~$F(\mathcal{T}') \subset \mathcal{S}'$. Then the induced functors 
	\[\overline{G}: \mathcal{S}/\mathcal{S}' \to \mathcal{T}/\mathcal{T}'~, \quad \overline{F}: \mathcal{T}/\mathcal{T}' \to \mathcal{S}/\mathcal{S}'\]
	are adjoint as well.
	\label{lmaadjunctionquot}
\end{lma}
\begin{proof}
	Let $\epsilon: F \circ G \to \id_{\mathcal{S}}, \eta: \id_{\mathcal{T}} \to G \circ F$ denote the unit and counit of the adjunction between $F$ and~$G$. The compositions $F \circ G$ and $G \circ F$ induce functors $\overline{F \circ G}: \mathcal{S}/\mathcal{S}' \to \mathcal{S}/\mathcal{S}'$ and $\overline{G \circ F}: \mathcal{T}/\mathcal{T}' \to \mathcal{T}/\mathcal{T}'$ and we have equalities $\overline{F} \circ \overline{G} = \overline{F \circ G} $ and~$\overline{G} \circ \overline{F} = \overline{G \circ F}$. If $q_{\mathcal{S}}: \mathcal{S} \to \mathcal{S}/\mathcal{S'}, q_{\mathcal{T}}: \mathcal{T} \to \mathcal{T}/\mathcal{T}'$ denote the respective Verdier quotient functors, we have in particular that $q_{\mathcal{S}} \circ F \circ G = \overline{F \circ G} = \overline{F} \circ \overline{G}$ and $q_{\mathcal{T}} \circ G \circ F = \overline{G \circ F} = \overline{G} \circ \overline{F}$. Using this we obtain maps $\overline{\epsilon}_A: \overline{F} \circ \overline{G} (A) \to A$ and $\overline{\eta}_B: B \to \overline{G} \circ \overline{F} (B)$ for all objects $A \in \mathcal{S}/\mathcal{S}', B \in \mathcal{T}/\mathcal{T}'$  by defining $\overline{\epsilon}_{A} := q_{\mathcal{S}}(\epsilon_A)$ and~$\overline{\eta}_B := q_{\mathcal{T}}(\eta_B)$. In fact, the collections of all $\overline{\epsilon}_{A}$ and $\overline{\eta}_{B}$ define natural transformations $\overline{\epsilon}:  \overline{F} \circ \overline{G} \to \id_{\mathcal{S}/\mathcal{S}'}$ and $\overline{\eta}: \id_{\mathcal{T}/\mathcal{T}'} \to \overline{G} \circ \overline{F}$ and we see that they satisfy the required counit-unit equations by considering the corresponding equations for $\epsilon, \eta$ and applying the functors~$q_{\mathcal{S}}, q_{\mathcal{T}}$.
\end{proof}

\begin{thm}
	The functors $\underline{\Ind}^G_H$ and $\underline{\Res}^G_H$ form an adjoint pair that satisfies the projection formula, in the sense of Definition \ref{defprojform}.
	\label{thmresindprojform}
\end{thm}
\begin{proof}
	The adjunction $\Ind^G_H \dashv \Res^G_H$ directly passes to an adjunction between $\Db(kH\Mod)$ and $\Db(kG\Mod)$ since both functors are exact. Both functors preserve projectives and therefore $\Ind^G_H (\mathrm{K^b}(kH\proj)) \subset \mathrm{K^b}(kG\proj) $ and $\Res^G_H(\mathrm{K^b}(kG\proj)) \subset \mathrm{K^b}(kH\proj)$. It follows from Lemma \ref{lmaadjunctionquot} that we obtain an induced adjunction $\underline{\Ind}^G_H \dashv \underline{\Res}^G_H$. 
	
	Furthermore, Frobenius reciprocity (see e.g.~\cite{carlson}*{Theorem 3.1}) tells us that there are natural isomorphisms in $kG\Mod$
	\[\Ind^G_H(L) \otimes M \cong \Ind^G_H(L \otimes \Res^G_H(M))\]
	and these descend to the stable category to give us natural isomorphisms
	\[\underline{\Ind}^G_H(L) \otimes M \cong \underline{\Ind}^G_H(L \otimes \underline{\Res}^G_H(M))~.\]
	This shows that the pair $(\underline{\Ind}^G_H,\underline{\Res}^G_H)$ satisfies the projection formula as desired.
\end{proof}

The following Lemma will be useful when we discuss the relative dimension of the induction functor.
\begin{lma}
	Let $M \in kH\Mod$. Then 
	\[\dim_k(\Ind^G_H(M)) = [G:H] \cdot \dim_k(M).\]
	\label{dimind}
\end{lma}
\begin{proof}
	There is an isomorphism of $kH$-modules
	\[\Res^G_H\left(\Ind^G_H(M)\right) \cong \Res^G_H\left(kG \otimes_{kH} M\right) \cong \left(\bigoplus_{G/H} kH \right) \otimes_{kH} M \cong \bigoplus_{G/H} M\]
	which proves the lemma as $\Res^G_H$ leaves dimensions intact.
\end{proof}

\subsubsection*{Relative dimension of restriction.}
We now consider the stable restriction functor
\[\underline{\Res}^G_H: kG\stab \to kH\stab~.\] 
We fix the Krull dimension as a dimension function for $kG\stab$ and~$kH\stab$. Recall that for $M \in kX\stab$ we have $\dim(\supp(M)) = c_X(M) - 1$ for~$X=G,H$. We begin with the following easy observation:
\begin{lma}
	Let $M \in kG\Mod$. Then 
	\[\mathrm{c}_H\left(\Res^G_H(M)\right) \leq \mathrm{c}_G(M)\]
	\label{reslesscomplexity}
\end{lma}
\begin{proof}
	Assume that $M \in kG\Mod$ has complexity $s$. The functor $\Res^G_H$ sends a minimal projective resolution $P_{\bullet} \to M$ to a projective resolution $\Res^G_H(P_{\bullet}) \to \Res^G_H(M)$ of $\Res^G_H(M)$ since it is exact and preserves projectives. A minimal projective resolution $Q_{\bullet} \to \Res^G_H(M)$ admits an injective chain map to $\Res^G_H(P_{\bullet}) \to \Res^G_H(M)$ by definition, and therefore we must have 
	\[\dim_k(Q_n) \leq \dim_k(\Res^G_H(P_n)) = \dim_k(P_n) \leq \kappa \cdot n^{s-1}\]
	as $M \in kG\Mod$ had complexity $s$. This implies that $\Res^G_H(M)$ has complexity~${\leq s}$.
\end{proof}

With a little more work we can now compute the relative dimension of~$\underline{\Res}^G_H$:
\begin{thm}
	Let $H \subset G$ be a subgroup such that $p$ divides $|H|$. Then $\underline{\Res}^G_H$ has relative dimension 0.
	\label{thmresreldim0}
\end{thm}
\begin{proof}
	It follows from Lemma \ref{reslesscomplexity} that for all objects $M \in kG\stab$, we have the inequality 
	\[\dim(\supp(\underline{\Res}^G_H)) \leq \dim(\supp(M))\] 
	and thus, if $\underline{\Res}^G_H$ has a relative dimension, it must be $\leq 0$. In order to prove the statement, it therefore suffices to show that there is an object $M_0$ of $kG\stab$ such that $\dim(\supp(\underline{\Res}^G_H(M_0))) \geq \dim(\supp(M_0))$ and therefore $\dim(\supp(\underline{\Res}^G_H(M_0))) = \dim(\supp(M_0))$. 
	
	Let $P \in \mathcal{V}_H(k)$ be a closed point (which exists as $p$ divides $|H|$) and look at $Q = \Spc\left(\underline{\Res}^G_H\right)(P) \in \mathcal{V}_G(k)$ which is closed as well since the map $\Spc\left(\underline{\Res}^G_H\right)$ is closed (see \cite{balmersep}*{Theorem 2.4 (b)}). Take $M_0 \in kG\stab$ such that $\supp(M_0) = \lbrace Q \rbrace$. This is possible as we can realize any subvariety as the support of a module, see \cite{bensonrep2}*{Chapter 5.9}, or more abstractly \cite{balmer2005spectrum}*{Corollary 2.17}. Note that this means that $\dim(\supp(M_0)) = 0$. We know that 
	\[\supp(\underline{\Res}^G_H(M_0)) = \left(\Spc\left(\underline{\Res}^G_H\right)\right)^{-1} \left(\supp(M_0)\right)\]
	which must contain the closed point $P$. Thus, 
	\[\dim(\supp(\underline{\Res}^G_H(M_0))) \geq 0 = \dim(\supp(M_0))~,\] 
	which finishes the proof.
\end{proof}

\begin{rem}
	Assume that $p \nmid |H|$. Then by Maschke's theorem $kH$ is semi-simple (see e.g.\ \cite{carlson}*{Theorem 1.7}), which implies that every finitely generated left $kH$-module is projective. Consequently, $kH\stab = 0$ and $\underline{\Res}^G_H(M) = 0$ for all modules $M \in kG\stab$. As $\dim(\supp(0)) = \dim(\emptyset) = -\infty$, the functor $\underline{\Res}^G_H$ does not have a relative dimension in this case.
\end{rem}

\subsubsection*{Relative dimension of induction.}
Let us consider the stable induction functor
\[\underline{\Ind}^G_H: kH\stab \to kG\stab\] 
next. Again, we fix the Krull dimension as a dimension function for $kH\stab$ and~$kG\stab$.
\begin{lma}
	Let $M \in kH\Mod$. Then
	\[\mathrm{c}_G\left(\Ind^G_H(M)\right) \leq \mathrm{c}_H(M)\]
	\label{indupperbound}
\end{lma}
\begin{proof}
	Let $M \in kH\Mod$ have minimal projective resolution $P_{\bullet} \to M$. Assume that $M$ has complexity $s$, then $\dim_k(P_n) \leq \kappa \cdot n^{s-1}$ for all $n$ and some constant $\kappa$. As $\Ind^G_H$ is an exact functor that preserves projectives, $\Ind^G_H(P_{\bullet}) \to \Ind^G_H(M)$ is a projective resolution of $\Ind^G_H(M)$. By Lemma \ref{dimind}, we have that 
	\[\dim_k \Ind^G_H(P_n) = [G:H] \dim_k(P_n) \leq [G:H] \kappa \cdot n^{s-1} \]
	from which it follows that a minimal projective resolution of $\Ind^G_H(M)$ has growth rate at most $s-1$, as it admits an injective chain map to $\Ind^G_H(P_{\bullet}) \to \Ind^G_H(M)$. Thus $\Ind^G_H(M)$ has complexity at most $s$.
\end{proof}

We now need two easy auxiliary lemmas concerning projective $kG$-modules. Recall that~$p=\mathrm{char}(k)$.
\begin{lma}[see \cite{carlson}*{Corollary 1.6}]
	Let $p^a$ be the exact power of $p$ dividing $|G|$ and $P$ a projective $kG$-module. Then $p^a$ divides $\dim_k(P)$.
	\label{projectivitycrit}
\end{lma}
\begin{proof}
	Let $S \leq G$ be a Sylow $p$-subgroup of order $p^a$. Then $\Res^G_S(P)$ is a projective $kS$-module. As $S$ is a $p$-group, projectivity and freeness of $kS$-modules coincide, so $\dim_k(\Res^G_S(P)) = \dim_k(P)$ is a multiple of $|S| = p^a$.
\end{proof}

\begin{lma}
	Let $H \leq G$ be a subgroup such that $p$ divides $|H|$. Then we have $\underline{\Ind}^G_H(k) \neq 0$, so in particular $\supp(\underline{\Ind}^G_H(k))$ is a non-empty closed subset of $V_G(k)$.
	\label{lmasuppindknonempty}
\end{lma}
\begin{proof}
	The module $\underline{\Ind}^G_H(k)$ being non-zero is equivalent to $\Ind^G_H(k)$ being non-projective. We know that $\Ind^G_H(k)$ is the permutation representation on the cosets $G/H$, which has dimension $[G:H]$. If $p^a$ is the exact power of $p$ dividing $|G|$, then assuming that $p$ divides $|H|$ tells us that $p^a \nmid [G:H] = \dim_k(\Ind^G_H(k))$ which implies that $\Ind^G_H(k)$ cannot be projective by Lemma \ref{projectivitycrit}.
\end{proof}

\begin{cor}
	Let $H \leq G$ be a subgroup such that $p$ divides $|H|$. Then $\underline{\Ind}^G_H$ has a relative dimension and it is $\leq 0$.
	\label{corindreldimex}
\end{cor}
\begin{proof}
	From Lemma \ref{indupperbound}, we see that $\underline{\Ind}^G_H(kH\stab_{(n)}) \subset kG\stab_{(n)}$ for all $n \in \mathbb{Z}_{\geq 0}$. Since we have $\dim(\supp(\underline{\Ind}^G_H(k))) \geq 0$ by Lemma \ref{lmasuppindknonempty}, it follows that $\dim\left(\underline{\Ind}^G_H\right) > -\infty$.
\end{proof}

\begin{thm}
	Let $H \leq G$ be a subgroup such that $p$ divides $|H|$. Then the functor $\underline{\Ind}^G_H$ has relative dimension $0$.
	\label{thmreldimind0}
\end{thm}
\begin{proof}
	By Theorem \ref{thmresindprojform}, Theorem \ref{thmresreldim0}, Lemma \ref{lmasuppindknonempty} and Corollary \ref{corindreldimex}, the assumptions of Corollary \ref{reldimsum} are satisfied. Therefore
	\[0 \leq \dim\left(\underline{\Res}^G_H\right) + \dim\left(\underline{\Ind}^G_H\right) = \dim\left(\underline{\Ind}^G_H\right)\]
	since we already know that the relative dimension of $\underline{\Res}^G_H$ is zero. Together with Corollary \ref{corindreldimex} this yields that the relative dimension of $\underline{\Ind}^G_H$ is $0$.
\end{proof}

\begin{rem}
	If  $p \nmid |H|$, then  $kH\stab = 0$ and $\underline{\Ind}^G_H$ is the inclusion of $0$ into $kG\stab$. This functor does not have a relative dimension as 
	\[\dim(\supp(0)) = \dim(\emptyset) = -\infty~.\]
\end{rem}

\small{\textsc{Sebastian Klein, Universiteit Antwerpen, Departement Wiskunde-Informatica, Middelheimcampus, Middelheimlaan 1, 2020 Antwerp, Belgium}\\
\textit{E-mail address:} \texttt{sebastian.klein@uantwerpen.be}}

\end{document}